\documentclass[10pt,reqno]{amsart}
\usepackage{amsmath,amscd,amssymb}
\usepackage{url}
\usepackage{faktor}
\usepackage{tikz}
\usepackage{graphicx}
\usepackage{subfigure}
\usepackage{color}
\usepackage[shortlabels]{enumitem}
\usepackage[pagebackref,colorlinks,citecolor=blue,linkcolor=magenta]{hyperref}

\usepackage[numbers, sort, comma, square]{natbib}

\definecolor{kth blue}{RGB}{26,84,166}
\definecolor{ist green}{RGB}{57,74,19}
\definecolor{easy green}{RGB}{60,179,113}

\makeatletter
\renewcommand*{\eqref}[1]{%
  \hyperref[{#1}]{\textup{\tagform@{\ref*{#1}}}}%
}
\makeatother

\usepackage{filecontents}

\makeatletter
    \let\@internalcite\cite
    \def\cite{\def\citeauthoryear##1##2{##1, ##2}\@internalcite}
    \def\shortcite{\def\citeauthoryear##1{##2}\@internalcite}
    \def\@biblabel#1{\def\citeauthoryear##1##2{##1, ##2}[#1]\hfill}
\makeatother

\theoremstyle{plain}
   \newtheorem{theorem}{Theorem}
   \newtheorem{proposition}[theorem]{Proposition}
   \newtheorem{lemma}[theorem]{Lemma}
   
\theoremstyle{definition}
   \newtheorem{definition}{Definition}
   \newtheorem{example}{Example}
\theoremstyle{remark}
 \newtheorem{remark}{Remark}

\newcommand{\A}{\mathcal{A}}
\newcommand{\GG}{\mathcal{G}}
\newcommand{\HH}{\mathcal{H}}

\newcommand{\I}{\mathcal{I}}
\newcommand{\J}{\mathcal{J}}
\newcommand{\M}{\mathcal{M}}

\newcommand{\qleftarrow}{ \! \leftarrow\!\!\! \circ  \,}
\newcommand{\qrightarrow}{\,\circ\!\!\! \rightarrow \! }

\newcommand\independent{\protect\mathpalette{\protect\independenT}{\perp}}
\def\independenT#1#2{\mathrel{\rlap{$#1#2$}\mkern2mu{#1#2}}}

\def\newop#1{\expandafter\def\csname #1\endcsname{\mathop{\rm
#1}\nolimits}}

\newop{an}
\newop{barren}
\newop{la}
\newop{sa}
\newop{pa}
\newop{de}
\newop{sp}
\newop{ant}
\newop{tail}
\newop{Dis}
\newop{dis}

\begin{document}
\title[Interventional Markov Equivalence for Mixed Graph Models]{Distributional Invariances and Interventional Markov Equivalence for Mixed Graph Models}

\author{Liam Solus}
\address{Matematik, KTH, SE-100 44 Stockholm, Sweden}
\email{solus@kth.se}

\date{\today}

\begin{abstract}
The invariance properties of interventional distributions relative to the observational distribution, and how these properties allow us to refine Markov equivalence classes (MECs) of DAGs, is central to causal DAG discovery algorithms that use both interventional and observational data.  
Here, we show how the invariance properties of interventional DAG models, and the corresponding refinement of MECs into interventional MECs, can be generalized to mixed graphical models that allow for latent cofounders and selection variables.  
We first generalize interventional Markov equivalence to all formal independence models associated to loopless mixed graphs. 
For ancestral graphs, we prove the resulting interventional MECs admit a graphical characterization generalizing that of DAGs.  
We then define interventional distributions for acyclic directed mixed graph models, and prove that this generalization aligns with the graphical generalization of interventional Markov equivalence given for the formal independence models. 
This provides a framework for causal model discovery via observational and interventional data in the presence of latent confounders that applies even when the interventions are uncontrolled. 
\end{abstract}

\maketitle
\thispagestyle{empty}

\section{Introduction}
\label{sec: introduction}

A major goal in statistics is to predict the effects of interventions amongst jointly distributed random variables.  
This is relevant in fields such as medicine, computational biology, and economics where practitioners are interested in altering predicted outcomes \cite{P00,RHB00,SGS01}.  
Oftentimes, the goal is to learn a causal DAG, which encodes a set of conditional independence (CI) relations and whose edges encode causal relations amongst the variables.  
However, using observational data alone, it is only possible to learn a causal DAG up to the CI relations it encodes, which are the same for all DAGs in its Markov equivalence class (MEC) \cite{AMP97}.  

By further sampling from interventional distributions, i.e.,~distributions produced by modifying variables within the system, identifiability of a causal DAG can be improved.  
For example, perfect (or hard) interventions \cite{EGS05}, which force the targeted variables to the values of independent variables, partition an MEC of DAGs into perfect $\I$-MECs \cite{HB12}. 
Similarly, general interventions, which do not necessarily force the targeted variables to become independent, partition an MEC of DAGs into (general) $\I$-MECs \cite{YKU18}.  
Recently, \citet{HB12} and \citet{YKU18} gave graphical characterizations of perfect $\I$-MECs and $\I$-MECs for DAGs, respectively.  
These characterizations play a key role in algorithms that use observational and interventional data to learn causal DAGs \cite{HB12,WSKU17,YKU18}. 

While DAG models are a natural starting point for causal inference, they make restrictive assumptions that can be alleviated by working with more general mixed graph models.  
For example, biologists are interested in learning protein synthesis networks that use both directed and bidirected arrows \cite{Sachs05}.  
Mixed graph models associate CI relations to a graph with edges that may be directed, bidirected, or undirected, and thereby incorporate latent confounders and selection variables into the model.  
The goal then is to relate the interventional and observational distributions so as to learn the signaling network exactly; that is, to learn exactly which edge types are in which locations.  
For DAG models this is done by studying the invariances in the factorization of the interventional distribution with respect to a DAG relative to the factorization for the observational distribution. 
In \cite{YKU18}, the authors showed that these invariance properties of DAG model factorizations are encoded entirely by combinatorial separation properties of DAGs.  
Thus, to generalize the theory of interventional Markov equivalence for DAG models to mixed graph models, there are two avenues to consider:  (1) generalize the distributional invariance properties of DAG models to models admitting similar factorization criteria, and (2) generalize the invariance properties as they are encoded via combinatorial separation statements in mixed graphs.  
Here, we will study both of these generalizations to the fullest extent possible (as permitted by the existing theory of mixed graph models).  
In doing so, we will provide a complete generalization of the theory of interventional Markov equivalence for DAG models (i.e., a theory unifying both (1) and (2)) to the family of directed ancestral graph models.  

To do so, we first generalize interventional Markov equivalence to all formal independence models for loopless mixed graphs.  
Under mild assumptions on the intervention targets, we extend the graphical characterization of interventional Markov equivalence of \cite{YKU18} to all ancestral graphs (AGs).  
We then define interventional distributions for acyclic directed mixed graph models (ADMGs) based on invariance properties of their factorization criterion.  
We prove that any distribution arising as the marginal of a post-intervention distribution Markov to a DAG over the observed variables will be an interventional distribution with respect to the ADMG that is the latent projection of the DAG onto the observed variables.
Finally, we prove that our generalization of interventional Markov equivalence for AGs via formal independence models and the probabilistic version for ADMGs coincide at their intersection; i.e., for all directed AGs. 
In this case, we require no assumptions on the intervention targets. 
In particular, the theory developed here extends the recent work of \cite{KJSB19} who studied interventional Markov equivalence when the interventions are \emph{controlled}; an assumption restricting the distributions that can be considered.
The theory presented here applies to any mixture of observational and interventional data, even for uncontrolled interventions.

\section{Formal Independence Models and Mixed Graphs} 
\label{sec: independence models and mixed graphs}

We first recall the necessary basics of mixed graph models, formal independence models, and interventions.  
All graph theory terms that are not defined in the following are defined in the appendix.
A \emph{mixed graph} is a graph $\GG = ([p],E)$ in which the set of edges $E$ contains a mixture of \emph{undirected} $i-j$, \emph{bidirected} $i \leftrightarrow j$, and \emph{directed} $i\rightarrow j$ edges.  
If it is ambiguous as to whether or not there is an arrowhead at the endpoint $i$ of an edge $e$, we will place an empty circle at the endpoint $i$; for example, $i \qrightarrow j$.  
A \emph{loopless mixed graph} (LMG) is a mixed graph that does not contain any \emph{loops}; i.e.,~edges which have both endpoints being the same node.   
We will call an LMG \emph{simple} if between any two nodes there is at most one edge.
We can define a notion of separation within an LMG that generalizes the notion of $d$-separation in DAGs \cite{SL14}:
Let $\GG = ([p],E)$ be an LMG and $C\subset[p]$. 
A path $\langle v_1,v_2,v_3\rangle$ in $\GG$ is a \emph{collider path} if the two edges on the path both have arrows pointing towards $v_2$.  
It is further called a \emph{v-structure} if $v_1$ and $v_3$ are non-adjacent in $\GG$.  
A path $\pi$ is \emph{$m$-connecting} given $C$ in $\GG$ if all of its collider subpaths $\langle v_1,v_2,v_3\rangle$ have $v_2\in C\cup\an_\GG(C)$ and all of its nodes that are not $v_2$ for some collider subpath are not in $C$.  
For two disjoint subsets $A,B\subset[p]$, we say $C$ \emph{$m$-separates} $A$ and $B$ in $\GG$ if there is no $m$-connecting path between a node in $A$ and a node in $B$ given $C$ in $\GG$.  

To an LMG $\GG = ([p],E)$, we associate a collection of triples $\J(\GG)$, where
$
\langle A, B \mid C\rangle \in \J(\GG)
$
if and only if $A$ and $B$ are $m$-separated given $C$ in $\GG$.  
We call the set of triples $\J(\GG)$ the \emph{(formal) independence model} for $\GG$.  
In general, $\GG$ is not uniquely determined by $\J(\GG)$, which leads to the notion of Markov equivalence:
\begin{definition}
\label{def: markov equivalence}
Two LMGs $\GG$ and $\HH$ are \emph{Markov equivalent}, denoted $\GG\approx\HH$, and belong to the same \emph{Markov equivalence class} (MEC) if $\J(\GG) = \J(\HH)$.  
\end{definition}
A distribution $\mathbb{P}$ over random variables $X_1,\ldots,X_p$ is \emph{Markov} with respect to $\GG$ if $\mathbb{P}$ entails $X_A\independent X_B\mid X_C$ whenever $\langle A, B\mid C\rangle\in \J(\GG)$.  
We let $\M(\GG)$ denote the collection of all density functions of distributions $\mathbb{P}$ that are Markov to $\GG$.  
If there exists $\mathbb{P}$ such that $\J(\GG)$ encodes precisely the set of CI relations entailed by $\mathbb{P}$ then we call $\J(\GG)$ a \emph{probabilistic independence model}.  
A number of families of LMGs yield well-studied probabilistic independence models.

The structure of an LMG $\GG = ([p],E)$ is intimately tied to the factorization of the distributions in $\M(\GG)$.  
The LMG $\GG$ is called a \emph{directed acyclic graph} (DAG) if it has only directed edges and no directed cycles.  
In this case, $f\in \M(\GG)$ if and only if it factorizes in a simple way:
\begin{equation}
\label{eqn:DAG}
f(x) = \prod_{i\in[p]}f(x_i \mid x_{\pa_\GG(i)}).
\end{equation}
In the case of DAGs, typically all members of an MEC only have a subset of their directed arrows in common.  
Thus, we can only learn a subset of the causal relations between nodes by sampling from the observational distribution.  
One way to improve identifiability of the causal DAG, i.e., learn more of the directed arrows, is to perform interventions on variables in the system and sample data from the resulting interventional distributions.   
Let $X = (X_1,\ldots,X_p)$ be a random vector with joint distribution $\mathbb{P}$ Markov to a DAG $\GG = ([p],E)$, and let $f^{(\emptyset)}$ denote its density function. 
Given a subset $I\subset[p]$, called an \emph{intervention target}, a distribution with density $f^{(I)}$ admitting a factorization
\begin{equation}
\label{eqn: interventional distribution}
f^{(I)}(x) = \prod_{i\in I}f^{(I)}(x_i \mid x_{\pa_\GG(i)})\prod_{i\notin I}f^{(\emptyset)}(x_i \mid x_{\pa_\GG(i)})
\end{equation}
is called an \emph{interventional distribution} (with respect to $\GG$ and $I$).  
The intervention is \emph{perfect} if $f^{(I)}(x_i \mid x_{\pa_\GG(i)}) = f^{(I)}(x_i)$ for all $i\in I$.  

Let $\I$ be a multiset of intervention targets, and $(f^{(I)})_{I\in\I}$ be a sequence of densities over $X = (X_1,\ldots,X_p)$ indexed by the elements of $\I$.  
For a DAG $\GG = ([p],E)$, define the set of \emph{intervention settings} for $\GG$ with respect to $\I$ as
\begin{equation}
\label{eqn: DAG interventional distributions}
\begin{split}
\M_\I(\GG) :=&\{(f^{(I)})_{I\in\I} \mid
\forall I,J\in\I : f^{(I)}\in\M(\GG), \mbox{ and }\\ 
&f^{(I)}(x_i\mid x_{\pa_\GG(i)}) = f^{(J)}(x_i\mid x_{\pa_\GG(i)}) \mbox{ for all } i\notin I\cup J\}.\\
\end{split}
\end{equation}
Since $\M_\I(\GG)$ represents the collection of all sequences of densities that can be generated from the interventions $\I$, \citet{YKU18} used this to formally define Markov equivalence of DAGs under general interventions:
\begin{definition}\cite[Definition 3.4]{YKU18}
\label{def: I-Markov equivalence of DAGs}
For the intervention targets $\I$, two DAGs $\GG$ and $\HH$ are \emph{$\I$-Markov equivalent} and belong to the same \emph{$\I$-Markov equivalence class} ($\I$-MEC) if $\M_\I(\GG) = \M_\I(\HH)$.
\end{definition} 

Notice that this definition, which relies on invariance properties of the factorization~\eqref{eqn:DAG}, is fundamentally different than the definition of Markov equivalence of DAGs given by Definition~\ref{def: markov equivalence}, which is phrased in terms of global Markov properties.  
However, the result of \cite[Theorem 3.9]{YKU18} demonstrates that $\I$-Markov equivalence of DAGs can also be defined using the global Markov property of DAGs.  
This suggests a generalization to mixed graph models, for which it is easier to work with global Markov properties since the corresponding factorization criteria generalizing \eqref{eqn:DAG} are not as simple (or simply unknown).  

\subsection{Interventional graphs and global $\I$-Markov properties } 
\label{subsec: interventional graphs and global I-markov properties}

We now recall the result of \cite{YKU18} so as to set the stage for our corresponding generalization.  

\begin{definition}
\label{def: interventional graph}
Let $\GG = ([p],E)$ be an LMG and $\I$ a collection of intervention targets.  
The \emph{interventional graph for $\I$} ($\I$-LMG) is the graph $\GG^\I$ with nodes $[p]\cup W_\I$, where
\[
W_\I :=\{\omega_I \mid I\in\I\backslash\{\emptyset\}\},
\]
and edges $E\cup E_\I$, where
\[
E_\I := \{\omega_I\rightarrow i \mid i\in I, I\in\I\backslash\{\emptyset\}\}.
\]
\end{definition}

	\begin{figure}
	\centering

\begin{tikzpicture}[thick,scale=0.3]
	

	 \node[circle, draw, fill=black!0, inner sep=1pt, minimum width=1pt] (1) at (-24,3) {$1$};
	 \node[circle, draw, fill=black!0, inner sep=1pt, minimum width=1pt] (2) at (-22,0) {$2$};
	 \node[circle, draw, fill=black!0, inner sep=1pt, minimum width=1pt] (3) at (-21,3) {$3$};
	 \node[circle, draw, fill=black!0, inner sep=1pt, minimum width=1pt] (4) at (-18,3) {$4$};
	 \node[circle, draw, fill=black!0, inner sep=1pt, minimum width=1pt] (5) at (-15,3) {$5$};
	 \node[circle, draw, fill=black!0, inner sep=1pt, minimum width=1pt] (6) at (-15,0) {$6$};
	 \node[circle, draw, fill=black!0, inner sep=1pt, minimum width=1pt] (7) at (-18,0) {$7$};

	 \node[circle, draw, fill=black!0, inner sep=1pt, minimum width=1pt] (g1) at (-9,3) {$1$};
	 \node[circle, draw, fill=black!0, inner sep=1pt, minimum width=1pt] (g2) at (-7,0) {$2$};
	 \node[circle, draw, fill=black!0, inner sep=1pt, minimum width=1pt] (g3) at (-6,3) {$3$};
	 \node[circle, draw, fill=black!0, inner sep=1pt, minimum width=1pt] (g4) at (-3,3) {$4$};
	 \node[circle, draw, fill=black!0, inner sep=1pt, minimum width=1pt] (g5) at (0,3) {$5$};
	 \node[circle, draw, fill=black!0, inner sep=1pt, minimum width=1pt] (g6) at (0,0) {$6$};
	 \node[circle, draw, fill=black!0, inner sep=1pt, minimum width=1pt] (g7) at (-3,0) {$7$};

	 \node[circle, draw, fill=black!100, inner sep=1pt, minimum width=1pt] (gw1) at (-10,5) {};
	 \node[circle, draw, fill=black!100, inner sep=1pt, minimum width=1pt] (gw14) at (-6,6) {};
	 \node[circle, draw, fill=black!100, inner sep=1pt, minimum width=1pt] (gw4) at (-1,5) {};
	 
	 \draw        (1) -- (2) ;
 	 \draw[->]   (2) -- (3) ;
 	 \draw[->]   (4) -- (3) ;
 	 \draw[->]   (4) -- (5) ;
 	 \draw[->]   (5) -- (6) ;
 	 \draw[->]   (6) -- (7) ;
 	 \draw[->]   (7) -- (4) ;

	 \draw        (g1) -- (g2) ;
 	 \draw[->]   (g2) -- (g3) ;
 	 \draw[->]   (g4) -- (g3) ;
 	 \draw[->]   (g4) -- (g5) ;
 	 \draw[->]   (g5) -- (g6) ;
 	 \draw[->]   (g6) -- (g7) ;
 	 \draw[->]   (g7) -- (g4) ;
 	 \draw[->]   (gw1) -- (g1) ;
 	 \draw[->]   (gw14) -- (g1) ;
 	 \draw[->]   (gw14) -- (g4) ;
	 \draw[->]   (gw4) -- (g4) ;
	 
 	 \node  at (-9,5.5) {$\omega_{\{1\}}$};
 	 \node  at (-4,6) {$\omega_{\{1,4\}}$};
 	 \node  at (0,5.5) {$\omega_{\{4\}}$};
 	 \node  at (-13.5,-1) {$\GG$};
 	 \node  at (1.5,-1) {$\GG^\I$};
\end{tikzpicture}
	\vspace{-0.2cm}
	\caption{An LMG $\GG$ with $\I$-LMG $\GG^\I$ for $\I = \{\{1\},\{4\},\{1,4\}\}$.  Note that $\GG$ is ribbonless graph (see the definition in subsection~\ref{subsec: ancestral and ribbonless graphs}), but $\GG^\I$ is not.}
	\label{fig: intervention example}
	\end{figure}
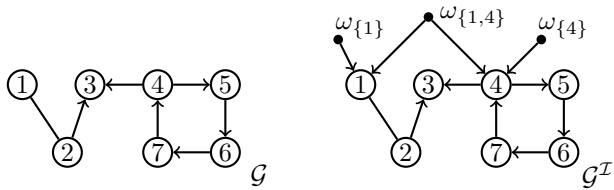
Figure~\ref{fig: intervention example} shows an LMG and its $\I$-LMG.\footnote{As in \cite{YKU18}, we treat interventions as parameters instead of random variables so as to avoid problems with faithfulness violations as discussed in \cite{MCM16}.  To make this distinction in our depictions of $\I$-LMGs, we draw nodes corresponding to random variables as open circles containing labels, and we draw interventional nodes in $W_\I$ as filled, black circles.}
\citet{VP91} proved two DAGs are Markov equivalent if and only if they have the same adjacencies and the same v-structures.  
The characterization of $\I$-MECs given by \citet{YKU18} extends this result via the combinatorics of $\I$-DAGs.  
When considered in the context of \cite{VP91}, their characterization frames $\I$-Markov equivalence for DAGs in terms of the global Markov property for DAGs.
In the following, if $\langle i,j,k\rangle$ is a v-structure in an $\I$-LMG $\GG^\I$ such that $i$ or $k$ is in $W_\I$, we call it an \emph{$\I$-v-structure}.

\begin{theorem}
\cite[Theorem 3.9]{YKU18}
\label{thm: YKU characterization}
Let $\I$ be a collection of interventional targets for which $\emptyset\in\I$.  
Two DAGs $\GG$ and $\HH$ are $\I$-Markov equivalent if and only if $\GG$ and $\HH$  are Markov equivalent and $\GG^\I$ and $\HH^\I$ have the same $\I$-v-structures.  
\end{theorem}

It is worth noting that \citet{MCM16} also consider $\I$-LMGs in their \emph{Joint Causal Inference} (JCI) framework, in which they use the $\I$-LMG to learn causal structure by applying classic causal discovery algorithms to the $\I$-LMG.  
On the other hand, they do not provide explicit graphical characterizations of $\I$-MECs as in Theorem~\ref{thm: YKU characterization}.  

By the result of \citet{VP91}, Theorem~\ref{thm: YKU characterization} equivalently states that two DAGs $\GG$ and $\HH$ are $\I$-Markov equivalent if and only if $\GG^\I$ and $\HH^\I$ are Markov equivalent.  
The former interpretation (stated in Theorem~\ref{thm: YKU characterization}) captures the intuition that intervening on nodes of a graph should ``lock'' more arrows in place; that is, determine more causal relations.  
However, the latter interpretation suggests a natural generalization of $\I$-Markov equivalence to formal independence models for LMGs, free of a factorization criterion, and based solely on global Markov properties:

\begin{definition}
\label{def: LMG I-Markov equivalence}
Let $\I$ be a collection of intervention targets.
Two LMGs $\GG$ and $\HH$ are \emph{$\I$-Markov equivalent} and belong to the same \emph{$\I$-Markov equivalence class} ($\I$-MEC) if
\[
\J(\GG^\I) = \J(\HH^\I).
\]
\end{definition}

\begin{remark}
\label{rmk: graphical invariance}
It may seem counterintuitive to define a notion of interventional equivalence based solely on the graph structure and without regard to the invariance properties of the underlying distribution.  
However, Theorem~\ref{thm: YKU characterization} in fact gives a completely graph-theoretical way of viewing these invariances.  
Recall that the $\I$-DAG $\GG^\I$ adds a new source node $\omega_I$ to $\GG$ such that $\omega_I\rightarrow i$ for all $i\in I$.  
Theorem~\ref{thm: YKU characterization} states that the intervention represented by the addition of this node will not change the causal mechanisms of nodes from which the point of intervention can be $d$-separated.  
For example, the causal mechanism $f(x_j\mid x_{\pa_\GG(j)})$ for any ancestor $j\notin I$ of a node $i\in I$ will remain unchanged by the intervention $I$, and this is captured graphically by the fact that $\omega_I$ is $d$-separated from $j$ in $\GG^\I$ given only $W_\I\setminus w_I$.  
In Section~\ref{sec: factorization criteria}, our choice of abstraction to this level will be further justified when we show that any two directed AGs are $\I$-Markov equivalent in terms of Definition~\ref{def: LMG I-Markov equivalence} if and only if they encode the same interventional settings from a standard causal perspective.  
\end{remark}

In the case of DAGs, notice that Theorem~\ref{thm: YKU characterization} would be trivial if we take Definition~\ref{def: LMG I-Markov equivalence} as the definition of $\I$-Markov equivalence instead of Definition~\ref{def: I-Markov equivalence of DAGs}.  
This is because an $\I$-DAG is again a DAG.  
However, as seen in Figure~\ref{fig: intervention example}, an $\I$-LMG $\GG^\I$ need not always be in the same subclass of LMGs as $\GG$.
This makes generalizing Theorem~\ref{thm: YKU characterization} with respect to Definition~\ref{def: LMG I-Markov equivalence} nontrivial for classes of LMGs other than DAGs.  
In particular, we would like to generalize Theorem~\ref{thm: YKU characterization} to one such family of LMGs, called ancestral graphs (AGs).  
Doing so will allow for modeling interventions in the same fashion as DAGs, but now in the presence of latent confounders and selection variables.  
A complete generalization of Theorem~\ref{thm: YKU characterization} to all AGs with arbitrary intervention targets is impossible (as we will see in Example~\ref{ex: doubly intervened}).  
However, it will extend to all AGs under reasonable assumptions on the intervention targets.  
Notice that the generalization given in Definition~\ref{def: LMG I-Markov equivalence} is for formal independence models, and hence any corresponding generalization of Theorem~\ref{thm: YKU characterization}, such as the one we will prove in Theorem~\ref{thm: main}, will be a purely mathematical result.  
To give such a result probabilistic meaning, it is also necessary to generalize the collection of intervention settings associated to a DAG to more general LMGs and relate such generalizations to Definition~\ref{def: LMG I-Markov equivalence}.  
This is done in Section~\ref{sec: factorization criteria} (specifically, Theorem~\ref{thm: coincide}), so as to completely generalize Theorem~\ref{thm: YKU characterization} to a more general family of LMGs; namely, the directed AGs. 

\subsection{Some subfamilies of LMGs} 
\label{subsec: ancestral and ribbonless graphs}
We now introduce families of LMGs, and some of their properties, that will be necessary to generalize Theorem~\ref{thm: YKU characterization}. 
A \emph{ribbon} in $\GG$ is a collider path $\langle i,j,k \rangle$ such that in $\GG$:
	\begin{enumerate}[1.]
		\item There is no endpoint-identical edge between $i$ and $k$; i.e., there is no $i\leftrightarrow k$ edge in the case $i\leftrightarrow j\leftrightarrow k$, no $i\rightarrow k$ edge in the case $i\rightarrow j \leftrightarrow k$, and no $i - k$ edge in the case $i\rightarrow j \leftarrow k$, and
		\item $j$ or a descendant of $j$ is an endpoint of an undirected edge or is on a directed cycle.
	\end{enumerate}
A \emph{ribbonless graph} (RG) is an LMG containing no ribbons.  
See Figure~\ref{fig: intervention example} and \cite{S12} for examples.
RGs are \emph{probabilistic}; i.e., $\J(\GG)$ is a probabilistic independence model.
In fact, each such model can be induced by marginalizing and conditioning a DAG model \cite{S12}.  
A subclass of RGs admitting this property \cite{RS02} are the ancestral graphs.
An LMG $\GG$ is an \emph{ancestral graph} (AG) if:
	\begin{enumerate}[1.]
		\item $\GG$ contains no directed cycles,
		\item whenever $\GG$ contains a bidirected edge $i\leftrightarrow j$, there is no directed path from $i$ to $j$ or $j$ to $i$ in $\GG$, and
		\item whenever $\GG$ contains an undirected edge $i-j$, neither $i$ nor $j$ has an arrowhead pointing towards it.
	\end{enumerate}

AGs were introduced by \cite{RS02} as a generalization of DAG models that are closed under marginalization and conditioning.  
They offer a means by which to model causation in the presence of latent confounders and selection variables.  
The bidirected edge $i\leftrightarrow j$ can be interpreted as representing a triple $i\leftarrow k \rightarrow j$ where the node $k$ corresponds to an unobserved random variable.  
Similarly, the undirected edge $i - j$ can be thought of as a triple $i\rightarrow k \leftarrow j$ in which the variable corresponding to $k$ has been conditioned upon (a selection variable).  
We call the endpoints of an undirected edge $x - y$ \emph{selection adjacent nodes}, and we let $\sa(\GG)$ denote the collection of selection adjacent nodes in $\GG$.

If an AG contains no undirected edges we call it a \emph{directed AG}.  
Directed AGs are precisely those AGs that are also \emph{acyclic directed mixed graphs} (ADMGs); i.e., graphs containing only directed or bidirected edges and no directed cycles.  
Hence, the results we derive in Section~\ref{sec: factorization criteria} on the invariance properties of factorization criterion for ADMGs will apply to all directed AGs.  
The two key differences between AGs and RGs are that RGs allow directed cycles and arrowheads pointing into selection adjacent nodes.  
In Figure~\ref{fig: hierarchy}, we depict the hierarchy of the subfamilies of LMGs studied in this paper.  
An extensive treatment of the many well-studied subfamilies of LMGs can be found in \cite{SL14}.  
	\begin{figure}
	\centering

\begin{tikzpicture}[thick,scale=0.3]
	
	 \node (h1) at (0,0) {LMGs};
	 \node (h1) at (2.5,0) {$\subsetneq$};
	 \node (h2) at (5,0) {RGs};
	 \node (h1) at (7.5,0) {$\subsetneq$};
	 \node (h3) at (10,0) {AGs};
	 \node (h1) at (12.5,0) {$\subsetneq$};
	 \node  (h4) at (17,0) {directed AGs};
	 \node  (h5) at (10,4) {ADMGs};
	 \node (h1) at (15,4) {$\subsetneq$};
	 \node   (h6) at (19,4) {DAGs};

	 \node[rotate=45] (h1) at (7,2) {$\subsetneq$};
	 \node[rotate=45] (h1) at (18,2) {$\subsetneq$};

\end{tikzpicture}
	\vspace{-0.2cm}
	\caption{The hierarchy (by inclusion) of the subfamilies of LMGs considered in this paper.}
	\label{fig: hierarchy}
	\end{figure}
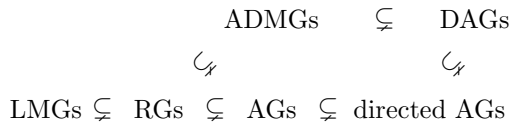

An LMG $\GG$ is called \emph{maximal} if adding any edge to $\GG$ changes the collection $\J(\GG)$.  
\citet[Theorem 5.1]{RS02} showed that every AG $\GG$ is contained in a unique \emph{maximal ancestral graph} (MAG) $\overline{\GG}$, and that $\overline{\GG}$ is produced from $\GG$ by adding bidirected arrows between the endpoints of the \emph{inducing paths} in $\GG$ (see the appendix for a definition).  
A similar result holds for RGs \cite{S12}.  
The existence of a unique maximal extension of a given RG will play a key role in the coming section.  

\section{$\I$-Markov Equivalence for Ancestral Graphs} 
\label{sec: I-Markov equivalence for MAGs}

With the definitions and observations of the previous section, we can now say exactly how Theorem~\ref{thm: YKU characterization} generalizes to AGs.
To understand the context in which the desired generalization holds, consider the following example:  
\begin{example}
\label{ex: doubly intervened}
Let $\GG$ and $\HH$ be the Markov equivalent AGs depicted in Figure~\ref{fig: doubly intervened}.  
As seen here, by taking $\I = \{\{1,4\},\{1,5\}\}$, the LMGs $\GG^\I$ and $\HH^\I$ have the same $\I$-v-structures.  
However, $\GG$ and $\HH$ are not $\I$-Markov equivalent.  
For example, $\langle\omega_{\{1,5\}}, \omega_{\{1,4\}}\mid 2\rangle\in\J(\GG^\I)$, but $\langle\omega_{\{1,5\}}, \omega_{\{1,4\}}\mid 2\rangle\notin\J(\HH^\I)$.
The induced subgraph of $\GG^\I$ on the node set $\{\omega_{\{1,4\}},\omega_{\{1,5\}},1,2\}$ is a special type of ribbon in $\GG^\I$ called a \emph{straight ribbon} \cite{SL14}.  
The underlying problem appears to be the ``double intervention'' on selection adjacent nodes, which results in a straight ribbon in the interventional graph.  
As we will see, assuming that we do not doubly-intervene on selection adjacent nodes is equivalent to assuming that $\GG^\I$ is ribbonless.  
This will be a sufficient condition for generalizing Theorem~\ref{thm: YKU characterization} to AGs.
\end{example}
	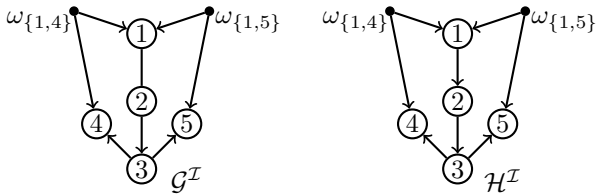
\begin{figure}[b!]
	\centering

\begin{tikzpicture}[thick,scale=0.3]
	
	 \node[circle, draw, fill=black!0, inner sep=1pt, minimum width=1pt] (h1) at (0,6) {$1$};
	 \node[circle, draw, fill=black!0, inner sep=1pt, minimum width=1pt] (h2) at (0,3) {$2$};
	 \node[circle, draw, fill=black!0, inner sep=1pt, minimum width=1pt] (h3) at (0,0) {$3$};
	 \node[circle, draw, fill=black!0, inner sep=1pt, minimum width=1pt] (h4) at (-2,2) {$4$};
	 \node[circle, draw, fill=black!0, inner sep=1pt, minimum width=1pt] (h5) at (2,2) {$5$};

	 \node[circle, draw, fill=black!100, inner sep=1pt, minimum width=1pt] (hw14) at (-3,7) {};
	 \node[circle, draw, fill=black!100, inner sep=1pt, minimum width=1pt] (hw15) at (3,7) {};

	 \node[circle, draw, fill=black!0, inner sep=1pt, minimum width=1pt] (g1) at (-14,6) {$1$};
	 \node[circle, draw, fill=black!0, inner sep=1pt, minimum width=1pt] (g2) at (-14,3) {$2$};
	 \node[circle, draw, fill=black!0, inner sep=1pt, minimum width=1pt] (g3) at (-14,0) {$3$};
	 \node[circle, draw, fill=black!0, inner sep=1pt, minimum width=1pt] (g4) at (-16,2) {$4$};
	 \node[circle, draw, fill=black!0, inner sep=1pt, minimum width=1pt] (g5) at (-12,2) {$5$};

	 \node[circle, draw, fill=black!100, inner sep=1pt, minimum width=1pt] (gw14) at (-17,7) {};
	 \node[circle, draw, fill=black!100, inner sep=1pt, minimum width=1pt] (gw15) at (-11,7) {};
 
	 \draw[->]   (h1) -- (h2) ;
 	 \draw[->]   (h2) -- (h3) ;
 	 \draw[->]   (h3) -- (h4) ;
 	 \draw[->]   (h3) -- (h5) ;
 	 \draw[->]   (hw14) -- (h1) ;
 	 \draw[->]   (hw14) -- (h4) ;
	 \draw[->]   (hw15) -- (h1) ;
 	 \draw[->]   (hw15) -- (h5) ;

	 \draw        (g1) -- (g2) ;
 	 \draw[->]   (g2) -- (g3) ;
 	 \draw[->]   (g3) -- (g4) ;
 	 \draw[->]   (g3) -- (g5) ;
 	 \draw[->]   (gw14) -- (g1) ;
 	 \draw[->]   (gw14) -- (g4) ;
	 \draw[->]   (gw15) -- (g1) ;
 	 \draw[->]   (gw15) -- (g5) ;
	 
 	 \node  at (-4.5,6.5) {$\omega_{\{1,4\}}$};
 	 \node  at (4.75,6.5) {$\omega_{\{1,5\}}$};
 	 \node  at (-18.5,6.5) {$\omega_{\{1,4\}}$};
 	 \node  at (-9.25,6.5) {$\omega_{\{1,5\}}$};
 	 \node  at (-12,-0.5) {$\GG^\I$};
 	 \node  at (2,-0.5) {$\HH^\I$};
\end{tikzpicture}
	\vspace{-0.2cm}
	\caption{Two $\I$-MAGs $\GG^\I$ and $\HH^\I$ such that $\GG$ and $\HH$ are Markov equivalent and $\GG^\I$ and $\HH^\I$ have the same $\I$-v-structures, but $\GG$ and $\HH$ are not $\I$-Markov equivalent.}
	\label{fig: doubly intervened}
	\end{figure}

Even with the assumption mentioned in Example~\ref{ex: doubly intervened}, we still capture useful causal models.  
For example, it does not rule out cases in which we can target single variables, such as when we generate interventional distributions for studying gene regulatory networks via targeted gene deletions using the CRISPR/CAS-9 system \cite{Dixit16}.  
It is also a trivial assumption if there are no selection variables; i.e., in the case of directed AGs.  

We further note that intervention settings that doubly-intervene on nodes arise naturally in practical contexts.  
For instance, the intervention setting considered in \cite{Sachs05} considers a multiset of intervention targets when learning a specific protein signaling network.  
Here, researchers are interested in modeling an interventional setting that includes different interventional experiments for the regulation of targeted molecules.  
The result is a model containing doubly-intervened nodes, as depicted in Figure~\ref{fig: Sachs} in the appendix.  
Similarly, doubly-intervened nodes such as depicted in Figure~\ref{fig: doubly intervened} can arise when the reagents for regulating expression target more than one molecule at a time.  
\begin{definition}
\label{def: doubly-intervened}
Let $\GG=([p],E)$ be an LMG and $\I$ a collection of intervention targets.  
We say that $\I$ \emph{doubly-intervenes} on a node $i\in[p]$ if there exist two distinct elements $I$ and $J$ in the multiset $\I$ such that $i\in I\cap J$.
\end{definition}

Recall that AGs are a subclass of RGs.
In the appendix, we show that, when intervening on an AG $\GG$, the target collection $\I$ does not doubly-intervene on any selection adjacent nodes in $\GG$ if and only if $\GG^\I$ is ribbonless.  
This is because the only ribbons one can create in an AG are straight ribbons. 
However, $\GG^\I$ may not be ribbonless if $\GG$ is an RG but not an AG.  
For example, if for the RG $\GG$ from Figure~\ref{fig: intervention example}, we intervene at $\I = \{\{4\},\{1,4\}\}$, then $\GG^\I$ is not ribbonless even though $\I$ does not doubly-intervene on any selection-adjacent nodes in $\GG$.
This is due to the \emph{cyclic ribbon} $\langle \omega_{\{1,4\}},4,\omega_{\{1\}}\rangle$.  

\citet{ARS09} showed that two MAGs are Markov equivalent if and only if they have the same \emph{colliders with order}.  
A collider $\langle i,j,k\rangle$ has \emph{order $0$} if $i$ and $k$ are not adjacent.  
For $t\geq0$, if we let $\mathfrak{D}_t$ denote the set of all \emph{triples of order $t$} in a given LMG $\GG$, then $\langle i,j,k\rangle\in\mathfrak{D}_{t+1}$ if it is not a triple with order $s<t+1$, and there is a path $\langle v_0,v_1,\ldots,v_m,b,c\rangle$ with either $\langle i,j,k\rangle = \langle v_m,b,c\rangle$ or $\langle k,j,i\rangle = \langle v_m,b,c\rangle$ such that 
\[
\langle v_0,v_1,v_2\rangle, \langle v_1,v_2,v_3\rangle, \ldots, \langle v_{m-1},v_m,b\rangle\in\bigcup_{s\leq t}\mathfrak{D}_s,
\]
$v_0$ and $c$ are nonadjacent in $\GG$, and $v_0,\ldots,v_m$ are parents of $c$.  
The path $\pi$ is a \emph{discriminating path} for $j$. 
For example, the path $\langle \omega_{\{1\}}, 1, 2, 3\rangle$ is a discriminating path for $2$ in the right-hand graph in Figure~\ref{fig: collider with order}. 

Just as how $\I$-MECs refine an MEC of DAGs by the presence of new v-structures (i.e., the $\I$-v-structures), we will see that $\I$-MECs of MAGs are defined by the presence new colliders with order.  
Given an LMG $\GG$, a \emph{$\I$-collider with order in $\GG^\I$} is a collider with order in $\GG^\I$ that is not a collider with order in $\GG$.  
An example is given in Figure~\ref{fig: collider with order}.
We can now state our main theorem in this section:
	\begin{figure}
	\centering

\begin{tikzpicture}[thick,scale=0.3]
	
 	 \node[circle, draw, fill=black!0, inner sep=1pt, minimum width=1pt] (t1) at (0,-8) {$1$};
	 \node[circle, draw, fill=black!0, inner sep=1pt, minimum width=1pt] (t2) at (3,-8) {$2$};
	 \node[circle, draw, fill=black!0, inner sep=1pt, minimum width=1pt] (t3) at (5,-5) {$3$};

	 \node[circle, draw, fill=black!0, inner sep=1pt, minimum width=1pt] (r1) at (13,-8) {$1$};
	 \node[circle, draw, fill=black!0, inner sep=1pt, minimum width=1pt] (r2) at (16,-8) {$2$};
	 \node[circle, draw, fill=black!0, inner sep=1pt, minimum width=1pt] (r3) at (18,-5) {$3$};

 	 \node[circle, draw, fill=black!100, inner sep=1pt, minimum width=1pt] (rwIJ) at (10,-7.8) {};

	 \draw[<->]   (t1) -- (t2) ;
 	 \draw[<->]   (t2) -- (t3) ;
 	 \draw[->]   (t1) -- (t3) ;

	 \draw[<->]   (r1) -- (r2) ;
 	 \draw[<->]   (r2) -- (r3) ;
 	 \draw[->]   (r1) -- (r3) ;
 	 \draw[->]   (rwIJ) -- (r1) ;
	 
 	 \node  at (9,-8.5) {$\omega_{\{1\}}$};
 	 \node  at (1,-5) {$\GG$};
 	 \node  at (14,-5) {$\GG^\I$};
 	
\end{tikzpicture}
	\vspace{-0.2cm}
	\caption{By intervening at $\I = \{\{1\}\}$, we produce the $\I$-collider $\langle 1,2,3\rangle$ with order $1$ and the $\I$-v-structure $\langle \omega_{\{1\}}, 1,2\rangle$ (i.e., an $\I$-collider with order $0$). Both are $\I$-colliders with order since they appear in $\GG^\I$ but not $\GG$.}
	\label{fig: collider with order}
	\end{figure}
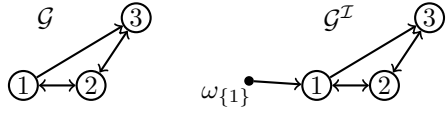

\begin{theorem}
\label{thm: main}
Let $\GG$ and $\HH$ be AGs, and let $\I$ be a collection of intervention targets such that no selection adjacent node is doubly-intervened in either $\GG$ or $\HH$.  
The following are equivalent:
\begin{enumerate}[1.]
	\item $\GG$ and $\HH$ are $\I$-Markov equivalent.
	\item $\overline{\GG}$ and $\overline{\HH}$ are $\I$-Markov equivalent.
	\item $\overline{\GG}$ and $\overline{\HH}$ are Markov equivalent and $\overline{\GG}^\I$ and $\overline{\HH}^\I$ have the same $\I$-colliders with order.
\end{enumerate}
\end{theorem}

Condition 3 in Theorem~\ref{thm: main} is actually a graphical characterization of $\I$-Markov equivalence.  
This is because there already exist graphical characterizations of Markov equivalence of MAGs, such as \cite[Theorem 3.2]{ARS09} or \cite[Theorem 2.1]{ZZL05}.  
Thus, even though $\overline{\GG}^\I$ and $\overline{\HH}^\I$ are not necessarily AGs, by condition 3 we can still apply these graphical characterizations and the condition on the $\I$-colliders with order in $\overline{\GG}^\I$ and $\overline{\HH}^\I$ to get graphical characterizations of $\I$-Markov equivalence for the AGs $\GG$ and $\HH$.  
For example, the two AGs $\GG$ and $\HH$ in Figure~\ref{fig: doubly intervened} are Markov equivalent with the same $\I$-colliders with order when we take $\I = \{\{1,4\}\}$ even though the left-hand graph $\GG$ is such that $\GG^\I$ is not ancestral.  
On the other hand, in the special case of Theorem~\ref{thm: main} in which no selection adjacent nodes are intervened upon in either $\GG^\I$ or $\HH^\I$, condition 3 is equivalent to saying that $\GG$ and $\HH$ are $\I$-Markov equivalent if and only if $\overline{\GG}^\I$ and $\overline{\HH}^\I$ are Markov equivalent MAGs, which can be deduced by the graphical characterizations of \cite[Theorem 3.2]{ARS09} or \cite[Theorem 2.1]{ZZL05}. 

For directed AGs, we note that Theorem~\ref{thm: main} places no restrictions on the intervention targets $\I$.  
In the coming section, we will see that directed AGs, and more generally ADMGs, admit a generalization of the interventional settings $\M_\I(\GG)$ via a known factorization criterion for such models.
This allows us to give probabilistic meaning to Theorem~\ref{thm: main}.  
Using this, we will recover a complete generalization of Theorem~\ref{thm: YKU characterization} to directed AGs.      

Finally, we note that Theorem~\ref{thm: main} can be used to describe $\I$-MECs for families of AGs in a number of contexts arising in practice:  
For instance, it applies when practitioners are able to design interventional experiments where each intervention target is a singleton.  
This case arises in genetics, where one can perform targeted gene deletions \cite{Dixit16}.  
On the other hand, Theorem~\ref{thm: main} is also applicable when practitioners are not able to place such restrictions on their intervention targets but are willing to model without the incorporation of selection variables.

\section{Factorization Criteria and $\I$-ADMGs} 
\label{sec: factorization criteria}

We have now generalized a characterization of $\I$-Markov equivalence for DAGs to AGs from the perspective of formal independence models as specified by Definition~\ref{def: LMG I-Markov equivalence}.  
However, it is typical to think of interventions as modifications to conditional factors of the density (i.e., modifications to causal mechanisms), and the formal independence model perspective obscures this. 
In Definition~\ref{def: I-Markov equivalence of DAGs}, $\I$-Markov equivalence for DAGs is defined in terms of modified densities, which is fundamentally different than the perspective taken in Definition~\ref{def: LMG I-Markov equivalence}.
The choice to use Definition~\ref{def: LMG I-Markov equivalence} can be made when there is no known factorization criterion for the graphical model that allows one to encode interventions via distributional invariances. 
However, since distributional invariance is how we statistically measure the effects of intervention, it is important that we check that Definition~\ref{def: LMG I-Markov equivalence} coincides with such distributional invariances when a factorization criterion for the model is known.
The most general family of LMGs for which we have a conditional factorization criterion is the family of ADMGs introduced in Subsection~\ref{subsec: ancestral and ribbonless graphs}, which contains all directed AGs.  
In this section, we extend the definition of the interventional settings $\M_\I(\GG)$ to ADMGs via their factorization criteria, and we characterize when two ADMGs $\GG$ and $\HH$ satisfy $\M_\I(\GG) = \M_\I(\HH)$.  
As a corollary, we will prove that for directed AGs, such an equality holds if and only if the two graphs satisfy the conditions in Theorem~\ref{thm: main} imposed by our definition of $\I$-Markov equivalence for formal independence models.

For an ADMG $\GG = ([p],E)$, we let
$
\A(\GG) :=\{ A \mid A = \an_\GG(A)\},
$
be the collection of all \emph{ancestrally closed sets} in $\GG$.  
The \emph{induced bidirected graph} $\GG_{\leftrightarrow}$ is the graph formed by removing all edges from $\GG$ that are not bidirected.  
Its path-connected components are called \emph{districts} (or \emph{$c$-components} \cite{TP02}).
We denote the collection of all districts in $\GG$ by $\dis(\GG)$, and we let $\dis_\GG(i)$ denote the district of $\GG$ containing the vertex $i\in[p]$.  
If $A\subset[p]$, we also let $\dis(A):=\dis(\GG\langle A\rangle)$ and $\dis_A(i) :=\dis_{\GG_A}(i)$.  
For $A\subset[p]$, let
\[
\barren_\GG(A) :=\{x\in A \mid \de_\GG(x)\cap A = \{x\}\}.
\]
We can then define the following functions on $A\subset[p]$:
First set $\Phi_\GG(\emptyset):=\emptyset$ and $\Psi_\GG^0(A):=A$.  Then set
\begin{equation*}
\begin{split}
\Phi_\GG(A)&:=\{H \mid H=\cap_{i\in H}\barren_\GG(\an_\GG(\dis_A(i))) \mbox{ and } H\neq\emptyset\}, \mbox{ and}\\
\Psi_\GG(A)&:=A\setminus\bigcup_{H\in\Phi_\GG(A)}H.\\
\end{split}
\end{equation*}
Applying the function $\Psi_\GG$ repeatedly produces a partition of $A$ into a collection of \emph{heads} denoted 
\[
[A]_\GG :=\bigcup_{k\geq 0}\Phi_\GG(\Psi_\GG^k(A)).
\]
Each head $H\in[A]_\GG$ is a path-connected subgraph of $\GG_{\leftrightarrow}$, and it satisfies $H = \barren_\GG(\an_\GG(H))$.  
For $H\in[A]_\GG$, we define the associated \emph{tail}
\[
\tail(H) := (\dis_{\an(H)}(H)\setminus H)\cup\pa_\GG(\dis_{\an(H)}(H)).
\]
\citet{R09} derived the following factorization criterion for ADMGs:
\begin{theorem}
\cite[Theorem 4]{R09}
\label{thm: richardson factorization}
A probability distribution $\mathbb{P}$ with density function $f$ is Markov to an ADMG $\GG$ if and only if for every $A\in\A(\GG)$, the marginal density $f(x_A)$ factors as
\begin{equation}
\label{eqn: ADMG factorization}
f(x_A) = \prod_{H\in[A]_\GG}f(x_H \mid x_{\tail(H)}).
\end{equation}
\end{theorem}

Using Theorem~\ref{thm: richardson factorization}, we can extend the notion of an interventional distribution captured in equation~\eqref{eqn: interventional distribution} to ADMGs: 

\begin{definition}
\label{def: ADMG interventional density}
A distribution with density function $f^{(I)}$ is called an \emph{interventional distribution} with respect to the intervention target $I$ and ADMG $\GG$ if for all $A\in\A(\GG)$
\begin{equation}
\label{eqn: ADMG interventional factorization}
\begin{split}
f(x_A) &= \prod_{H\in[A]_\GG : I\cap H \neq \emptyset}f^{(I)}(x_H \mid x_{\tail(H)})\prod_{H\in[A]_\GG : I\cap H = \emptyset}f^{(\emptyset)}(x_H \mid x_{\tail(H)}),
\end{split}
\end{equation}
where $f^{(\emptyset)}$ is the density of a distribution Markov to $\GG$. 
\end{definition}

The collection of sequences of interventional distributions $(f^{(I)})_{I\in\I}$ that can be generated by intervening on a collection of targets $\I$ with respect to an observational density Markov to an ADMG $\GG$ is then
\begin{equation}
\label{eqn: ADMG interventional density sets}
\begin{split}
\M_\I(\GG) := \{&(f^{(I)})_{I\in\I} \mid \mbox{$\forall I,J\in\I$ } \mbox{ and $A\in\A(\GG)$: $f^{(I)}\in\M(\GG)$ and } \\
	& f^{(I)}(x_H\mid x_{\tail(H)}) = f^{(J)}(x_H\mid x_{\tail(H)}) \mbox{ $\forall H\in[A]_\GG$: $H\cap(I\cup J) = \emptyset$}\}.
\end{split}
\end{equation}
A proof of this fact is given in the appendix.

\begin{remark}
\label{rmk: causal connection}
In Theorem~\ref{thm: causal alignment}, we will show that our definition for general interventions with respect to an ADMG (Definition~\ref{def: ADMG interventional density}) and our corresponding generalization of the collection $\M_\I(\GG)$ to ADMGs capture the intuitive notion of intervention given by interpreting a distribution Markov to an ADMG as the marginal over the observed variables in a distribution Markov to a causal DAG.
To prove Theorem~\ref{thm: causal alignment}, we will first characterize when $\M_\I(\GG) = \M_\I(\HH)$ for two ADMGs $\GG$ and $\HH$.
\end{remark}

Definition~\ref{def: ADMG interventional density} and \eqref{eqn: ADMG interventional density sets}  are the natural extensions of equations~\eqref{eqn: interventional distribution} and~\eqref{eqn: DAG interventional distributions} to the family of ADMGs.  
With these definitions, we can generalize the \emph{$\I$-Markov property for DAGs} introduced in \cite[Definition 3.6]{YKU18} by simply passing from $d$-separation to $m$-separation:

\begin{definition}[$\I$-Markov Property]
\label{def: ADMG local I-Markov properties}
Let $\I$ be a collection of intervention targets such that $\emptyset\in\I$ and $(f^{(I)})_{I\in\I}$ a set of strictly positive distributions over $X_1,\ldots,X_p$.  
For an ADMG $\GG = ([p], E)$, we say that $(f^{(I)})_{I\in\I}$ satisfies the \emph{$\I$-Markov property} with respect to $\GG^\I$ if and only if
\begin{enumerate}[1.]
	\item $X_A \independent X_B \mid X_C$ in $f^{(I)}$ for all $I\in\I$ and any $A,B,C\subset[p]$ for which $\langle A,B \mid C\rangle \in \J(\GG)$.  
	\item $f^{(\emptyset)}(X_B\mid X_C) = f^{(I)}(X_B\mid X_C)$ for any $I\in\I$ and $B,C\subset[p]$ such that \\$\langle B,\{\omega_I\} \mid C\cup W_\I\setminus\{\omega_I\}\rangle\in\J(\GG^\I)$.  
\end{enumerate}
\end{definition}

\begin{remark}
\label{rmk: causal I-Markov property}
\citet[Definition 1]{KJSB19} recently introduced an alternative $\I$-Markov property that is more restrictive than Definition~\ref{def: ADMG local I-Markov properties}.  
The original definition of \citet{YKU18}, and Definition~\ref{def: ADMG local I-Markov properties}, are designed to characterize exactly the invariances between observational and interventional distributions detectable in their factorizations, whereas the definition of \citet{KJSB19} is motivated by a generalization of Pearl's do-calculus to soft interventions \cite{CB19}.  
Consequently, their $\I$-Markov property, which we will call the \emph{controlled $\I$-Markov property}, is tied to additional invariances which are exhibited when the interventions are \emph{controlled}; meaning that if $I,J\in\I$ both target the node $j$, then they change the causal mechanism associated to $j$ in the exact same way.  
That is, $f^{(I)}(X_j\mid X_{\pa_\GG{(j)}}) = f^{(J)}(X_j\mid X_{\pa_\GG{(j)}})$.  
It follows that the controlled $\I$-Markov property imposes more invariances than the $\I$-Markov property.
In Section~\ref{app: comparing}, we show for any ADMG $\GG$ and collection of intervention targets $\I$, the collection of interventional settings satisfying the controlled $\I$-Markov property with respect to $\GG$ is a (generally strict) subset of those satisfying the $\I$-Markov property.  
Even for DAGs, this inclusion is often strict.  
The assumption that all interventions are controlled can be restrictive in, for example, biological contexts when using different reagents to regulate protein expression: Two reagents may target the same molecule but regulate its expression in very different ways.  
The theory we provide here does not assume interventions are controlled and, hence, is applicable to these more general contexts. 
\end{remark}

\citet[Proposition 3.8]{YKU18} showed that for a DAG $\GG$, if $\emptyset\in\I$, then $(f^{(I)})_{I\in\I}$ is in $\M_\I(\GG)$ if and only if $(f^{(I)})_{I\in\I}$ satisfies the $\I$-Markov property with respect to $\GG^\I$.  
The following theorem generalizes this result to all ADMGs, and hence characterizes interventional model equivalence for ADMGs from the perspective of distributional invariances in their associated factorization criteria.

\begin{theorem}
\label{thm: local I-Markov equivalence characterization}
Suppose $\emptyset\in\I$ and $\GG$ is an ADMG.  
Then $(f^{(I)})_{I\in\I}\in\M_\I(\GG)$ if and only if $(f^{(I)})_{I\in\I}$ satisfies the $\I$-Markov property with respect to $\GG^\I$.  
\end{theorem}

\subsection{Interventions in ADMGs causally} 
\label{subsec: causal perspective}
Since ADMGs can be used to represent causal DAGs with latent confounders, we would like to see that our definition of interventional setting given in~\eqref{eqn: ADMG interventional density  sets} captures the intuitive of notion interventional setting arising from this perspective.  
To formalize this, we first recall the connection between causal DAGs with latent confounders and ADMGs.  
\begin{definition}
\label{def: latent projection}
Let $\GG$ be a DAG with vertex set $[p]\sqcup L$, where the vertices in $[p]$ are observed and those in $L$ are latent, and $\sqcup$ denotes a disjoint union.
The \emph{latent projection} $\GG([p])$ is a mixed graph with vertex set $[p]$ such that for every pair of distinct vertices $i,j\in[p]$
\begin{enumerate}[1.]
	\item $\GG([p])$ contains an edge $i\rightarrow j$ if there exists a directed path in $\GG$ $i\rightarrow\cdots\rightarrow j$ on which every non-endpoint vertex is in $L$,  and
	\item $\GG([p])$ contains an edge $i\leftrightarrow j$ if there exists a path in $\GG$ between $i$ and $j$ such that the endpoints are all non-colliders in $L$ and such that the edge adjacent to $i$ and the edge adjacent to $j$ both have arrowheads pointing towards $i$ and $j$, respectively.  
\end{enumerate}
\end{definition}
Since $\GG = ([p]\sqcup L,E)$ is a DAG, we see that $\GG([p])$ is an ADMG.  
In the following, given $f\in\M(\GG)$, we let $\tilde{f}:=\int_LfdL$ denote the marginal density of $f$ over $[p]$.  
It is well-known that $\tilde{f}$ is Markov to $\GG([p])$ (see for instance \cite[Theorem~7.1]{RS02}).

Given an ADMG $\HH = ([p],E)$ we can always produce a DAG $\GG = ([p]\sqcup L,E^\prime)$ such that $\HH = \GG([p])$ by replacing each bidirected arrow $i\leftrightarrow j$ in $\HH$ with $i\leftarrow \ell_{i,j}\rightarrow j$ and letting $L :=\{\ell_{i,j}: i\leftrightarrow j \text{ an edge in } \HH\}$.  
Thus, given a DAG $G = ([p]\sqcup L,E)$ and an interventional setting $(f^{(I)})_{I\in\I}\in \M_\I(\GG)$ for $\I = \{\emptyset,I_1,\ldots,I_K\}$ with $I_k\in[p]$ for all $k\in[K]$ we would like to see that $(\tilde{f}^{(I)})_{I\in\I}\in\M_\I(\GG([p]))$.  
This would imply that the invariance properties used to define interventional densities and interventional settings for ADMGs in Definition~\ref{def: ADMG interventional density} and \eqref{eqn: ADMG interventional density  sets} align with the standard notion given by Definition~\ref{eqn: interventional distribution} and the causal interpretation of ADMGs as latent projections of causal DAGs.
To prove this, we use the fact that $d$-separation statements on the observed variables in the DAG $\GG = ([p]\sqcup L,E)$ are in one-to-one correspondence with the $m$-separation statements in the latent projection $\GG([p])$.  
A proof of this fact appears first in \cite{V93} and later on in \cite[Proposition 4]{RERS17}.
Combining this fact with our Theorem~\ref{thm: local I-Markov equivalence characterization}, we can prove that our definition of interventional setting for ADMGs aligns with the causally intuitive notion.
\begin{theorem}
\label{thm: causal alignment}
Let $\GG = ([p]\sqcup L,E)$ be a DAG, where $[p]\sqcup L$ is a disjoint union.  
Let $\I$ be a collection of intervention targets satisfying $\emptyset\in\I$ and $I\subset [p]$ for all $I\in\I$.  
If $(f^{(I)})_{I\in\I}\in\M_\I(\GG)$ then $(\tilde{f}^{(I)})_{I\in\I}\in\M_\I(\GG([p]))$.  
\end{theorem}

\subsection{Directed ancestral graphs} 
\label{subsec: directed AGs}
Our last goal is to connect the interventional settings for ADMGs with the combinatorial characterization of $\I$-Markov equivalence for formal independence models given in Theorem~\ref{thm: main}.
Theorem~\ref{thm: local I-Markov equivalence characterization} says two ADMGs $\GG$ and $\HH$ satisfy $\M_\I(\GG) = \M_\I(\HH)$ if and only if they have the same $\I$-Markov properties.  
Theorem~\ref{thm: main} characterizes $\I$-MECs of AGs via global Markov properties.  
Our last theorem says these two generalizations coincide for LMGs that are both AGs and ADMGs; i.e., for directed AGs.  
This unifies all of our previous results, and it completely generalizes Theorem~\ref{thm: YKU characterization} to a context in which we can incorporate latent confounders without assuming controlled interventions.  

\begin{theorem}
\label{thm: coincide}
Let $\GG$ and $\HH$ be two directed AGs and let $\I$ be a collection of intervention targets.  
Suppose $\emptyset\in\I$.  Then the following are equivalent:
\begin{enumerate}[1.]
	\item $\GG$ and $\HH$ are $\I$-Markov equivalent,
	\item $\overline{\GG}$ and $\overline{\HH}$ are $\I$-Markov equivalent,
	\item $\GG$ and $\HH$ are Markov equivalent and $\overline{\GG}^\I$ and $\overline{\HH}^\I$ have the same $\I$-colliders with order, and
	\item $\M_\I(\GG) = \M_\I(\HH)$.  
\end{enumerate}
\end{theorem}

\begin{remark}
\label{rmk: refined}
\citet{KJSB19} defined a notion of $\I$-Markov equivalence by saying that two ADMGs, $\GG$ and $\HH$, are $\I$-Markov equivalent if and only if the set of distributions satisfying the controlled $\I$-Markov property with respect to $\GG$ and $\I$ and $\HH$ and $\I$ are the same.  
Under the assumption of controlled interventions, they characterized this phenomenon in a similar fashion as to our Theorem~\ref{thm: coincide}.  
Since, as noted in Remark~\ref{rmk: causal I-Markov property}, the assumption of controlled interventions requires that (in general strictly) more invariances hold among the distributions in an interventional setting satisfying the controlled $\I$-Markov property, the $\I$-Markov equivalence classes studied by \citet{KJSB19} should refine the $\I$-Markov equivalence classes identified by Theorem~\ref{thm: coincide}.  
In Section~\ref{app: comparing}, we show that this is indeed the case, and we give a graphical characterization of how this refinement occurs.  
While finer $\I$-Markov equivalence classes is of course desirable from a causal learning perspective, if the controlled assumption is too restrictive for the application at hand, one can now use the slightly larger $\I$-MECs characterized by Theorem~\ref{thm: coincide}.
\end{remark}

Theorem~\ref{thm: coincide} also shows that $\I$-Markov equivalence of directed AGs corresponds to invariance properties with respect to the factorization criterion in Theorem~\ref{thm: richardson factorization}.  
Therefore, these invariance properties of factorization criteria also hold for the interventional settings studied by \citet{KJSB19}.  
In this way, Theorem~\ref{thm: coincide} yields a complete theory of general interventions in directed AG models; i.e., a theory in which intervention is understood both from the perspective of Markov properties and in terms of distributional invariances with respect to the model factorization criterion. 
However, the specific invariances arising from the controlled assumption are not immediately visible in this factorization criterion.  
It would be interesting to know if controlled interventional settings admit an even more refined factorization criterion that exhibits all of their invariance properties. 

It is also interesting to note that Theorem~\ref{thm: coincide} yields a different characterization of $\I$-Markov equivalence given any characterization of Markov equivalence of AGs.  
Currently, there are three known characterizations of Markov equivalence for AGs: one due to \citet{ARS09}, another due to \citet{ZZL05}, and a third by \citet{SR97}.  
Applying Theorem~\ref{thm: coincide} to the first, we see that $\M_\I(\GG) = \M_\I(\HH)$ if and only if $\overline{\GG}^\I$ and $\overline{\HH}^\I$ have the same adjacencies and colliders with order.  Applying the second one, Theorem~\ref{thm: coincide} implies that $\M_\I(\GG) = \M_\I(\HH)$ if and only if $\overline{\GG}^\I$ and $\overline{\HH}^\I$ have the same $\I$-colliders with order and the same minimal collider paths.  Applying the third recovers yet another characterization. 

\section{Comparing Alternative $\I$-Markov Properties} 
\label{app: comparing}
\citet{KJSB19} recently investigated characterizations of $\I$-Markov equivalence under general interventions as well.  
To do so, they introduced a slightly different $\I$-Markov property than that of \citet[Definition 3.6]{YKU18} or our generalization thereof (Definition~\ref{def: ADMG local I-Markov properties}).  
Instead, their $\I$-Markov property, which we will refer to as the \emph{controlled $\I$-Markov property} is the following:
\begin{definition}
\cite[Definition 1]{KJSB19}
\label{def: causal I-Markov property}
Let $\I$ be a collection of intervention targets and $(f^{(I)})_{I\in\I}$ a set of strictly positive distributions over $X_1,\ldots,X_p$.  
We say that $(f^{(I)})_{I\in\I}$ satisfies the \emph{controlled $\I$-Markov property} with respect to an ADMG $\GG = ([p],E)$ if and only if
\begin{enumerate}[1.]
	\item $X_A \independent X_B \mid X_C$ in $f^{(I)}$ for all $I\in\I$ and any $A,B,C\subset[p]$ for which $\langle A,B \mid C\rangle \in \J(\GG)$, and 
	\item $f^{(I)}(X_B\mid X_C) = f^{(J)}(X_B\mid X_C)$ for any two $I,J\in\I$ and $B,C\subset[p]$ such that 
	\[
	\langle B,\{\omega_{I\triangle J} \mid C\cup W_{\I\triangle\I}\setminus\{\omega_{I\triangle J}\}\rangle\in\J(\GG^{\I\triangle\I}),
	\]
	where $\I\triangle\I :=\{ I\triangle J : I,J\in\I\}$ is the set of all symmetric differences of the sets in $\I$.  

\end{enumerate}
\end{definition}
We refer to Definition~\ref{def: causal I-Markov property} as the controlled $\I$-Markov property since it is motivated by a recent generalization of Pearl's do-calculus to the case of soft interventions \cite{CB19}.  
\citet[Corollary $1$]{KJSB19} showed that property (2) of the controlled $\I$-Markov property is a combined version of the `do-do' and `do-see' rules from this generalized do-calculus whenever the interventions $\I$ are assumed to be controlled.  
In this way, the controlled $\I$-Markov property is a more restrictive version of the $\I$-Markov property presented in Definition~\ref{def: ADMG local I-Markov properties} since it is inherently tied to the assumption that the interventions are controlled.  
We note also that the phrasing of Definition~\ref{def: causal I-Markov property} is not the same as that of \cite[Definition $1$]{KJSB19}, but instead uses the language and notation developed in this paper. 
However, the equivalence of these two definitions follows immediately from \cite[Proposition $1$]{KJSB19}.

Let $\M_\I^{(c)}(\GG)$ denote the set of all $(f^{(I)})_{I\in\I}$ satisfying the controlled $\I$-Markov property with respect to $\GG$. 
\citet{KJSB19} defined two graphs $\GG$ and $\HH$ to be $\I$-Markov equivalent whenever $\M_\I^{(c)}(\GG) = \M_\I^{(c)}(\HH)$.  
We will refer to this as \emph{controlled $\I$-Markov equivalence}.
The two differences between the $\I$-Markov property defined in \cite[Definition 3.6]{YKU18} and Definition~\ref{def: ADMG local I-Markov properties} and the controlled $\I$-Markov property is the choice for property (2) and the fact that the controlled $\I$-Markov property does not assume $\emptyset\in\I$.  
As is seen in \cite{YKU18}, the assumption that $\emptyset\in\I$ can be readily dropped by treating a fixed interventional distribution as the observational distribution.  
So we can then assume (without loss of generality) that $\emptyset\in\I$.  
Then the only difference between the $\I$-Markov property studied here and the controlled $\I$-Markov property is the choice of property (2).  

Property (2) of the $\I$-Markov property, as originally introduced in \cite[Definition 3.6]{YKU18}, was defined to characterize the invariances in the interventional distribution 
\[
f^{(I)}(x) = \prod_{i\in I}f^{(I)}(x_i \mid x_{\pa_\GG(i)})\prod_{i\in[p]\setminus I}f^{(\emptyset)}(x_i \mid x_{\pa_\GG(i)})
\]
relative to a given observational distribution 
\[
f^{(\emptyset)}(x) = \prod_{i\in[p]}f^{(\emptyset)}(x_i \mid x_{\pa_\GG(i)})
\]
that is Markov to a DAG $\GG = ([p],E)$.  
Namely, property (2) of the $\I$-Markov property characterizes the invariances $f^{(I)}(x_j \mid x_{\pa_\GG(j)}) = f^{(J)}(x_j \mid x_{\pa_\GG(j)})$ in the causal mechanisms that arise whenever $j\notin I\cup J$ for two interventions $I,J\in\I$.  
The $\I$-Markov property that we define in Defintion~\ref{def: ADMG local I-Markov properties} directly generalizes this, and in doing so, property (2) of Definition~\ref{def: ADMG local I-Markov properties} captures all invariances in the causal mechanisms identified by the factorization criterion of Theorem~\ref{thm: richardson factorization}; that is to say, $f^{(I)}(x_H \mid x_{\tail(H)}) = f^{(J)}(x_H \mid x_{\tail(H)})$ whenever $H\cap(I\cup J) = \emptyset$.  

On the other hand, the controlled $\I$-Markov property was developed to capture the invariances in a recently introduced do-calculus for soft interventions \cite{CB19}.  
As noted in \cite[Corollary 1]{KJSB19}, property (2) of their controlled $\I$-Markov property corresponds to invariances that are guaranteed by the rules of this do-calculus for soft interventions when we assume that all interventions are controlled (as defined in Remark~\ref{rmk: causal I-Markov property}).  
The assumption that all interventions are controlled can be restrictive in, for example, biological contexts as described in Remark~\ref{rmk: causal I-Markov property}.   
Consequently, the controlled $\I$-Markov property is satisfied only by interventional settings $(f^{(I)})_{I\in\I}$ that satisfy (generally) more invariances than the invariances characterized by property (2) of the $\I$-Markov property studied in this paper.
Since the $\I$-Markov property studied here and in \cite{YKU18} is not tied to controlled interventions, the characterizations of $\I$-Markov equivalence given in Theorem~\ref{thm: coincide} apply to strictly more interventional settings than those captured by the controlled $\I$-Markov property.
We summarize these observations with the following proposition.  
To help emphasize these differences between these two $\I$-Markov properties, we include the proof here, along with an illustrative example in Figure~\ref{fig: minimal example}. 
\begin{proposition}
\label{prop: comparing properties}
Let $\I$ be a collection of intervention targets and $\emptyset\in\I$.  
For an ADMG $\GG = ([p],E)$, the $\I$-Markov property is weaker than the controlled $\I$-Markov property; that is,
\[
\M_\I^{(c)}(\GG) \subset \M_\I(\GG).
\]
Moreover, there exists DAGs $\GG$ and collections of intervention targets $\I$ for which
\[
\M_\I^{(c)}(\GG) \subsetneq \M_\I(\GG).
\]
\end{proposition}

\begin{proof}
It is straightforward to see the containment $\M_\I^{(c)}(\GG) \subset \M_\I(\GG)$ by taking $\J = \emptyset$ in property (2) of Definition~\ref{def: causal I-Markov property}.  
As for seeing that there exist cases of strict containment, pick $i,j\in[p]$ and a DAG $\GG = ([p],E)$ for which $i$ and $j$ are not adjacent in $\GG$, $\pa_\GG(i) = \pa_\GG(j)$, and $\an_\GG(i)\neq \pa_\GG(i)$. 
Consider then the intervention targets 
\[
\I = \{
\emptyset, 
I := \pa_\GG(i)\cup\{i\}, 
J := \pa_\GG(j)\cup\{j\}
\}.
\]
Then 
\[
\I\triangle\I = \{
\emptyset, 
I, 
J,
I\triangle J = \{i,j\}
\}.
\]
A minimal example of such a DAG $\GG$ and the set $\I$ is given in Figure~\ref{fig: minimal example}. 

Now let $v\in\an_\GG(i)\setminus\pa_\GG(i)$.
Since $\pa_\GG(i)$ is the Markov blanket of $\omega_{I\triangle J}$ in $\GG^{\I\triangle\I}$, then $\langle v,\omega_{I\triangle J} \mid \pa_\GG(i) \rangle \in\J(\GG^{\I\triangle\I})$.
On the other hand, $\langle v,\omega_{I} \mid \pa_\GG(i) \rangle, \langle v,\omega_{J} \mid \pa_\GG(i) \rangle\notin\J(\GG^{\I})$.  
This is because $v$ is ancestor of both $i$ and $j$, and so conditioning on $\pa_\GG(i) = \pa_\GG(j)$ induces $d$-connecting paths between $v$ and $\omega_I$ and $v$ and $\omega_J$ in $\GG^\I$ (and $\GG^{\I\triangle\I}$, in fact).  
By applying the constructions in the proof of \cite[Lemma 4]{KJSB19} to these $d$-connecting paths we can produce an interventional setting $(f^{(I)})_{I\in\I}\in\M_\I(\GG)$ such that
\begin{equation*}
\begin{split}
f^{(\emptyset)}(X_v \mid X_{\pa_\GG(i)}) &\neq f^{(I)}(X_v \mid X_{\pa_\GG(i)}),\\
f^{(\emptyset)}(X_v \mid X_{\pa_\GG(i)}) &\neq f^{(J)}(X_v \mid X_{\pa_\GG(i)}). \\
\end{split}
\end{equation*}
Moreover, by appropriately choosing the error terms as described in the proof of \cite[Lemma 4]{KJSB19}, we can ensure
\[
f^{(I)}(X_v \mid X_{\pa_\GG(i)}) \neq f^{(I)}(X_v \mid X_{\pa_\GG(i)}). \mbox{\,\,\,\,\,\,\,\,\, \,\,}
\]
Hence, $(f^{(I)})_{I\in\I}\notin\M_\I^{(c)}(\GG)$, and we conclude that $\M_\I^{(c)}(\GG)\subsetneq\M_\I(\GG)$. 
\end{proof}

	\begin{figure}
	\centering

\begin{tikzpicture}[thick,scale=0.3]
	
	 \node[circle, draw, fill=black!0, inner sep=1pt, minimum width=1pt] (h1) at (0,6) {$1$};
	 \node[circle, draw, fill=black!0, inner sep=1pt, minimum width=1pt] (h2) at (0,3) {$2$};
	 \node[circle, draw, fill=black!0, inner sep=1pt, minimum width=1pt] (h3) at (-2,0) {$3$};
	 \node[circle, draw, fill=black!0, inner sep=1pt, minimum width=1pt] (h4) at (2,0) {$4$};

	 \node[circle, draw, fill=black!100, inner sep=1pt, minimum width=1pt] (hw14) at (-3,4) {};
	 \node[circle, draw, fill=black!100, inner sep=1pt, minimum width=1pt] (hw15) at (3,4) {};

	 \node[circle, draw, fill=black!0, inner sep=1pt, minimum width=1pt] (g1) at (15,6) {$1$};
	 \node[circle, draw, fill=black!0, inner sep=1pt, minimum width=1pt] (g2) at (15,3) {$2$};
	 \node[circle, draw, fill=black!0, inner sep=1pt, minimum width=1pt] (g3) at (13,0) {$3$};
	 \node[circle, draw, fill=black!0, inner sep=1pt, minimum width=1pt] (g4) at (17,0) {$4$};

	 \node[circle, draw, fill=black!100, inner sep=1pt, minimum width=1pt] (gw14) at (12,4) {};
	 \node[circle, draw, fill=black!100, inner sep=1pt, minimum width=1pt] (gw15) at (18,4) {};
	 \node[circle, draw, fill=black!100, inner sep=1pt, minimum width=1pt] (gw145) at (15,-3) {};
 
	 \draw[->]   (h1) -- (h2) ;
 	 \draw[->]   (h2) -- (h3) ;
 	 \draw[->]   (h2) -- (h4) ;
 	 \draw[->]   (hw14) -- (h2) ;
 	 \draw[->]   (hw14) -- (h3) ;
	 \draw[->]   (hw15) -- (h2) ;
 	 \draw[->]   (hw15) -- (h4) ;

	 \draw[->]   (g1) -- (g2) ;
 	 \draw[->]   (g2) -- (g3) ;
 	 \draw[->]   (g2) -- (g4) ;
 	 \draw[->]   (gw14) -- (g2) ;
 	 \draw[->]   (gw14) -- (g3) ;
	 \draw[->]   (gw15) -- (g2) ;
 	 \draw[->]   (gw15) -- (g4) ;
 	 \draw[->]   (gw145) -- (g3) ;
 	 \draw[->]   (gw145) -- (g4) ;
	 
 	 \node  at (-3.75,4.5) {$\omega_{I}$};
 	 \node  at (3.75,4.5) {$\omega_{J}$};
 	 \node  at (11.25,4.5) {$\omega_{I}$};
 	 \node  at (18.75,4.5) {$\omega_{J}$};
 	 \node  at (16.75,-3.5) {$\omega_{I\triangle J}$};
 	 \node  at (-3,-3.5) {$\GG^\I$};
 	 \node  at (11,-3.5) {$\GG^{\I\triangle\I}$};
\end{tikzpicture}
	\caption{A DAG $\GG$ and set of interventions $\I = \{\emptyset, \{2,3\},\{2,4\}\}$ for which the collection of interventional settings $\M_\I^{(c)}(\GG)$ satisfying the controlled $\I$-Markov property (Definition~\ref{def: causal I-Markov property}) is a strict subset of the collection of interventional settings $\M_\I(\GG)$ satisfying the $\I$-Markov property (Definition~\ref{def: ADMG local I-Markov properties} or \protect\cite[Definition 3.6]{YKU18}).}
	\label{fig: minimal example}
	\end{figure}
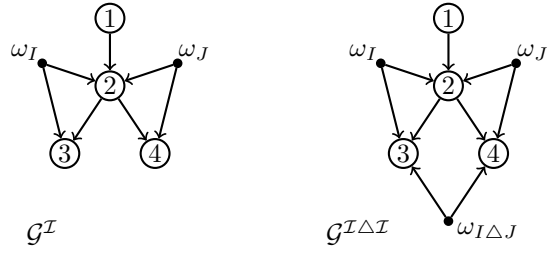

It is well-known that given an ADMG $\GG = ([p],E)$, a directed MAG $M_\GG = ([p],E)$ can be constructed that maintains the independence and ancestral relations amongst the variables $[p]$.  
In other words, one can choose a representative of the Markov equivalence class of $\GG$ that is a MAG with the same ancestral relations as $\GG$. 
\citet[Theorem 2]{KJSB19} gives a characterization of when two directed MAGs representing a causal graph are $\I$-Markov equivalent (with respect to the controlled $\I$-Markov property) under the assumption that the collection of interventions are controlled: 
\begin{theorem}
\cite[Theorem 2]{KJSB19}
\label{thm: kocaoglu characterization}
Two ADMGs $\GG$ and $\HH$ are $\I$-Markov equivalent for a set of controlled interventions $\I$ if and only if for the directed MAGs $M_\GG^{\I\triangle\I}$ and $M_\HH^{\I\triangle\I}$:
\begin{enumerate}
	\item $M_\GG^{\I\triangle\I}$ and $M_\HH^{\I\triangle\I}$ have the same skeleton,
	\item $M_\GG^{\I\triangle\I}$ and $M_\HH^{\I\triangle\I}$ have the same v-structures, and 
	\item If $\pi$ is a discriminating path for a node $i$ in both $M_\GG^{\I\triangle\I}$ and $M_\HH^{\I\triangle\I}$ then $i$ is a collider on $\pi$ in one graph if and only if it is a collider on the other.  
\end{enumerate}
\end{theorem}
It then follows from Theorem~\ref{thm: main} and the characterization of Markov equivalence for MAGs given by \cite{SR97} that two directed AGs $\GG$ and $\HH$ satisfy $\M_\I^{(c)}(\GG) = \M_\I^{(c)}(\HH)$ for a set of controlled interventions $\I$ if and only if they are $\I\triangle\I$-Markov equivalent according to Definition~\ref{def: LMG I-Markov equivalence}.  
Hence, when we are willing to assume controlled interventions, Theorem~\ref{thm: kocaoglu characterization} implies that we can partition the $\I$-Markov equivalence classes given by Theorem~\ref{thm: coincide} into finer $\I\triangle\I$-Markov equivalence classes.  
This, of course makes intuitive sense, since Theorem~\ref{thm: kocaoglu characterization} assumes the interventions are controlled, and hence imposes more invariances than needed in Theorem~\ref{thm: coincide}.  
An example of this refinement is given in Figure~\ref{fig: refined}.
	\begin{figure}
	\centering

\begin{tikzpicture}[thick,scale=0.3]
	
	 \node[circle, draw, fill=black!0, inner sep=1pt, minimum width=1pt] (h1) at (0,0) {$1$};
	 \node[circle, draw, fill=black!0, inner sep=1pt, minimum width=1pt] (h2) at (3,0) {$2$};
	 \node[circle, draw, fill=black!0, inner sep=1pt, minimum width=1pt] (h3) at (5,3) {$3$};

	 \node[circle, draw, fill=black!100, inner sep=1pt, minimum width=1pt] (hwI) at (1,4) {};
	 \node[circle, draw, fill=black!100, inner sep=1pt, minimum width=1pt] (hwJ) at (6,5.5) {};

	 \node[circle, draw, fill=black!0, inner sep=1pt, minimum width=1pt] (g1) at (13,0) {$1$};
	 \node[circle, draw, fill=black!0, inner sep=1pt, minimum width=1pt] (g2) at (16,0) {$2$};
	 \node[circle, draw, fill=black!0, inner sep=1pt, minimum width=1pt] (g3) at (18,3) {$3$};

	 \node[circle, draw, fill=black!100, inner sep=1pt, minimum width=1pt] (gwI) at (14,4) {};
	 \node[circle, draw, fill=black!100, inner sep=1pt, minimum width=1pt] (gwJ) at (19,5.5) {};

	 \node[circle, draw, fill=black!0, inner sep=1pt, minimum width=1pt] (t1) at (0,-8) {$1$};
	 \node[circle, draw, fill=black!0, inner sep=1pt, minimum width=1pt] (t2) at (3,-8) {$2$};
	 \node[circle, draw, fill=black!0, inner sep=1pt, minimum width=1pt] (t3) at (5,-5) {$3$};

	 \node[circle, draw, fill=black!100, inner sep=1pt, minimum width=1pt] (twI) at (1,-4) {};
	 \node[circle, draw, fill=black!100, inner sep=1pt, minimum width=1pt] (twJ) at (6,-2.5) {};
	 \node[circle, draw, fill=black!100, inner sep=1pt, minimum width=1pt] (twIJ) at (-3,-7.8) {};

	 \node[circle, draw, fill=black!0, inner sep=1pt, minimum width=1pt] (r1) at (13,-8) {$1$};
	 \node[circle, draw, fill=black!0, inner sep=1pt, minimum width=1pt] (r2) at (16,-8) {$2$};
	 \node[circle, draw, fill=black!0, inner sep=1pt, minimum width=1pt] (r3) at (18,-5) {$3$};

	 \node[circle, draw, fill=black!100, inner sep=1pt, minimum width=1pt] (rwI) at (14,-4) {};
	 \node[circle, draw, fill=black!100, inner sep=1pt, minimum width=1pt] (rwJ) at (19,-2.5) {};
 	 \node[circle, draw, fill=black!100, inner sep=1pt, minimum width=1pt] (rwIJ) at (10,-7.8) {};

	 \draw[<-]   (h1) -- (h2) ;
 	 \draw[->]   (h2) -- (h3) ;
 	 \draw[->]   (h1) -- (h3) ;
 	 \draw[->]   (hwI) -- (h1) ;
 	 \draw[->]   (hwI) -- (h3) ;
	 \draw[->]   (hwJ) -- (h3) ;

	 \draw[<->]   (g1) -- (g2) ;
 	 \draw[<->]   (g2) -- (g3) ;
 	 \draw[->]   (g1) -- (g3) ;
 	 \draw[->]   (gwI) -- (g1) ;
 	 \draw[->]   (gwI) -- (g3) ;
	 \draw[->]   (gwJ) -- (g3) ;

	 \draw[<-]   (t1) -- (t2) ;
 	 \draw[->]   (t2) -- (t3) ;
 	 \draw[->]   (t1) -- (t3) ;
 	 \draw[->]   (twI) -- (t1) ;
 	 \draw[->]   (twI) -- (t3) ;
	 \draw[->]   (twJ) -- (t3) ;
 	 \draw[->]   (twIJ) -- (t1) ;

	 \draw[<->]   (r1) -- (r2) ;
 	 \draw[<->]   (r2) -- (r3) ;
 	 \draw[->]   (r1) -- (r3) ;
 	 \draw[->]   (rwI) -- (r1) ;
 	 \draw[->]   (rwI) -- (r3) ;
	 \draw[->]   (rwJ) -- (r3) ;
 	 \draw[->]   (rwIJ) -- (r1) ;
	 
 	 \node  at (.1,3.75) {$\omega_{I}$};
 	 \node  at (6.75,6) {$\omega_{J}$};
 	 \node  at (13.1,3.75) {$\omega_{I}$};
 	 \node  at (19.75,6) {$\omega_{J}$};
 	 \node  at (.1,-4.25) {$\omega_{I}$};
 	 \node  at (6.75,-2) {$\omega_{J}$};
 	 \node  at (13.1,-4.25) {$\omega_{I}$};
 	 \node  at (19.75,-2) {$\omega_{J}$};
 	 \node  at (-4,-8.5) {$\omega_{I\triangle J}$};
 	 \node  at (9,-8.5) {$\omega_{I\triangle J}$};
\end{tikzpicture}
	\vspace{-0.2cm}
	\caption{The two MAGs are in the same $\I$-MEC for $\I = \{\emptyset, I = \{1,3\}, J = \{3\}\}$, but they are in distinct $\I\triangle\I$-MECs under the controlled intervention assumption.  This is because $\I\triangle\I =\{\emptyset, I, J, I\triangle J = \{1\}\}$, and the node $\omega_{I\triangle J}$ induces a new $\I$-collider with order $\langle 1,2,3\rangle$ in precisely one graph.}
	\label{fig: refined}
	\end{figure}

\section{Discussion}
\label{sec: discussion}
Here, we generalized the notion of $\I$-Markov equivalence to all formal independence models associated to loopless mixed graphs, and we gave a combinatorial characterization of this equivalence for ancestral graphs.  
We also generalized the interventional settings $\M_\I(\GG)$ for DAGs to ADMGs via their factorization criterion, and characterized when $\M_\I(\GG) = \M_\I(\HH)$.  
When the ADMG is also an AG, and we have access to the observed distribution, we showed that this equality holds precisely when $\GG$ and $\HH$ are $\I$-Markov equivalent from the perspective of formal independence models.  
This yields a theory for modeling interventional Markov equivalence in the presence of latent confounders via directed AGs that is applicable even when the interventions are uncontrolled.  
While the work here assumes access to the observational distribution, a natural follow-up project is to extend it to the purely interventional setting and then apply this theory to causal discovery algorithms that use observational and interventional data.  
It would also be of interest to extend this work to other families of LMGs, such as cyclic graphs. 

	\begin{figure}[b!]
	\centering

\begin{tikzpicture}[thick,scale=0.45]
	
	 \node[circle, draw, fill=black!0, inner sep=1pt, minimum width=1pt] (h1) at (0,0) {$1$};
	 \node[circle, draw, fill=black!0, inner sep=1pt, minimum width=1pt] (h2) at (3,0) {$2$};
	 \node[circle, draw, fill=black!0, inner sep=1pt, minimum width=1pt] (h3) at (3,3) {$3$};
	 \node[circle, draw, fill=black!0, inner sep=1pt, minimum width=1pt] (h4) at (0,3) {$4$};

	 \node[circle, draw, fill=black!0, inner sep=1pt, minimum width=1pt] (g1) at (10,0) {$1$};
	 \node[circle, draw, fill=black!0, inner sep=1pt, minimum width=1pt] (g2) at (13,0) {$2$};
	 \node[circle, draw, fill=black!0, inner sep=1pt, minimum width=1pt] (g3) at (13,3) {$3$};
	 \node[circle, draw, fill=black!0, inner sep=1pt, minimum width=1pt] (g4) at (10,3) {$4$};
 
	 \draw[<->]   (h1) -- (h2) ;
 	 \draw[<->]   (h2) -- (h3) ;
 	 \draw[<->]   (h3) -- (h4) ;
 	 \draw[->]   (h2) -- (h4) ;
 	 \draw[->]   (h3) -- (h1) ;

	 \draw[<->]   (g1) -- (g2) ;
 	 \draw[<->]   (g2) -- (g3) ;
 	 \draw[<->]   (g3) -- (g4) ;
 	 \draw[<->]   (g1) -- (g4) ;
 	 \draw[->]   (g2) -- (g4) ;
 	 \draw[->]   (g3) -- (g1) ;
	 
 	 \node  at (1.5,-1.5) {$\GG$};
 	 \node  at (11.5,-1.5) {$\overline{\GG}$};
\end{tikzpicture}
	\caption{An ancestral graph $\GG$ and the unique MAG containing it.}
	\label{fig: hard}
	\end{figure}

It is further interesting to note that, while the methods presented here and in \cite{KJSB19,YKU18} develop a theory for refining Markov equivalence classes by way of general interventions, in some cases, these methods are weaker than the methods applicable only to perfect interventions.  
For instance, if one wishes to distinguish between a given ancestral graph and the unique (Markov equivalent) MAG containing it, we cannot use any of the methods developed for general interventions.
However, we can use those existing only for perfect interventions.  
For example, the graph on the left-hand-side of Figure~\ref{fig: hard} is an ancestral graph for which the graph on the right-hand-side is the unique MAG containing it.  

In Figure~\ref{fig: hard intervened} we see on the top row that, using the methods for model interventions via $\I$-graphs studied here and in \cite{KJSB19,YKU18}, we cannot distinguish the two graphs.
This is because their $\I$-graphs have the same $\I$-colliders with order.  
On the other hand, when modeling a perfect intervention, we remove all arrows pointing into the targeted node.  
In the bottom line, of Figure~\ref{fig: hard intervened}, we see that by removing the arrows corresponding to the same intervention (assumed now to be perfect), we recover two ancestral graphs with different colliders with order.  
Hence, the methods for modeling perfect interventions can further distinguish graphs in the same Markov equivalence class.  
A careful investigation of what methods can be used for perfect and general interventions when attempting to distinguish between different nonmaximal elements of a Markov equivalence class would be of interest.  
	\begin{figure}
	\centering

\begin{tikzpicture}[thick,scale=0.45]
	
	 \node[circle, draw, fill=black!0, inner sep=1pt, minimum width=1pt] (h1) at (0,0) {$1$};
	 \node[circle, draw, fill=black!0, inner sep=1pt, minimum width=1pt] (h2) at (3,0) {$2$};
	 \node[circle, draw, fill=black!0, inner sep=1pt, minimum width=1pt] (h3) at (3,3) {$3$};
	 \node[circle, draw, fill=black!0, inner sep=1pt, minimum width=1pt] (h4) at (0,3) {$4$};

	 \node[circle, draw, fill=black!100, inner sep=1pt, minimum width=1pt] (hw14) at (5,.4) {};

	 \node[circle, draw, fill=black!0, inner sep=1pt, minimum width=1pt] (g1) at (13,0) {$1$};
	 \node[circle, draw, fill=black!0, inner sep=1pt, minimum width=1pt] (g2) at (16,0) {$2$};
	 \node[circle, draw, fill=black!0, inner sep=1pt, minimum width=1pt] (g3) at (16,3) {$3$};
	 \node[circle, draw, fill=black!0, inner sep=1pt, minimum width=1pt] (g4) at (13,3) {$4$};

	 \node[circle, draw, fill=black!100, inner sep=1pt, minimum width=1pt] (gw14) at (18,.4) {};

	 \node[circle, draw, fill=black!0, inner sep=1pt, minimum width=1pt] (r1) at (0,-8) {$1$};
	 \node[circle, draw, fill=black!0, inner sep=1pt, minimum width=1pt] (r2) at (3,-8) {$2$};
	 \node[circle, draw, fill=black!0, inner sep=1pt, minimum width=1pt] (r3) at (3,-5) {$3$};
	 \node[circle, draw, fill=black!0, inner sep=1pt, minimum width=1pt] (r4) at (0,-5) {$4$};

	 \node[circle, draw, fill=black!0, inner sep=1pt, minimum width=1pt] (t1) at (13,-8) {$1$};
	 \node[circle, draw, fill=black!0, inner sep=1pt, minimum width=1pt] (t2) at (16,-8) {$2$};
	 \node[circle, draw, fill=black!0, inner sep=1pt, minimum width=1pt] (t3) at (16,-5) {$3$};
	 \node[circle, draw, fill=black!0, inner sep=1pt, minimum width=1pt] (t4) at (13,-5) {$4$};

	 \draw[<->]   (h1) -- (h2) ;
 	 \draw[<->]   (h2) -- (h3) ;
 	 \draw[<->]   (h3) -- (h4) ;
 	 \draw[->]     (h2) -- (h4) ;
 	 \draw[->]     (h3) -- (h1) ;
 	 \draw[->]     (hw14) -- (h2) ;

	 \draw[<->]   (g1) -- (g2) ;
 	 \draw[<->]   (g2) -- (g3) ;
 	 \draw[<->]   (g3) -- (g4) ;
 	 \draw[<->]   (g1) -- (g4) ;
 	 \draw[->]     (g2) -- (g4) ;
 	 \draw[->]     (g3) -- (g1) ;
 	 \draw[->]     (gw14) -- (g2) ;

 	 \draw[<->]   (r3) -- (r4) ;
 	 \draw[->]     (r2) -- (r4) ;
 	 \draw[->]     (r3) -- (r1) ;

 	 \draw[<->]   (t3) -- (t4) ;
 	 \draw[<->]   (t1) -- (t4) ;
 	 \draw[->]     (t2) -- (t4) ;
 	 \draw[->]     (t3) -- (t1) ;
	 
 	 \node  at (6,0) {$\omega_{\{2\}}$};
 	 \node  at (19,0) {$\omega_{\{2\}}$};
 	 \node  at (1.5,-1.5) {$\GG^\I$};
 	 \node  at (14.5,-1.5) {$\overline{\GG}^{\I}$};
\end{tikzpicture}
	\caption{By intervening in the graphs in Figure~\ref{fig: hard} with the intervention target $\I = \{\{2\}\}$, we see that the methods applicable to general interventions (top row) yield two AGs with $\J(\GG^\I) = \J(\overline{\GG}^\I)$, and hence the graphs $\GG$ and $\overline{\GG}$ are $\I$-Markov equivalent under general interventions.  On the other hand, if we assume the interventions are perfect, we can remove the arrows pointing into node $2$, resulting in two graphs with different global Markov properties.  Hence, under perfect interventions, $\GG$ and $\overline{\GG}$ are not $\I$-Markov equivalent.}
	\label{fig: hard intervened}
	\end{figure}

\subsection*{Acknowledgements}
Liam Solus was supported by a Starting Grant (No. 2019-05195) from Vetenskapsr\aa{}det, and the Wallenberg Autonomous Systems and Software Program (WASP). 



\bibliography{interventions}
\bibliographystyle{abbrvnat}

\appendix

\section{Preliminary Definitions and Results on LMGs} 

\subsection{Graph Theory Definitions} 
The following is a complete dictionary of the graph theoretical terms used in the manuscript.  
A \emph{mixed graph} is a graph $\GG = ([p],E)$ in which the set of edges $E$ contains a mixture of \emph{undirected} $i-j$, \emph{bidirected} $i \leftrightarrow j$, and \emph{directed} $i\rightarrow j$ edges.  
A mixed graph $\GG$ is called \emph{simple} if there exists at most one edge of any type between each pair of nodes in $\GG$. 
For each edge-type, we call $i$ and $j$ the \emph{endpoints} of the edge, and we say $i$ and $j$ are \emph{adjacent} $\GG$ if they are the two endpoints of any single edge in $\GG$. 
We call $i\in[p]$ a \emph{source} if all edges with endpoint $i$ have no arrowhead at $i$, and we call $i$ a \emph{sink} if every edge with endpoint $i$ has an arrowhead at $i$. 
If it is ambiguous as to whether or not there is an arrowhead at the endpoint $i$ of an edge $e$, we will place an empty circle at the endpoint $i$; for example $i \qrightarrow j$.  
If $A\subset [p]$, the \emph{induced subgraph of $\GG$ on $A$}, denoted $\GG\langle A \rangle$, is the graph with node set $A$ and all edges in $E$ which have both endpoints in $A$.  
A \emph{loopless mixed graph} (LMG) is a mixed graph that does not contain any \emph{loops}; i.e., edges which have both endpoints being the same node.  
For an LMG $\GG = ([p],E)$, and $i\in[p]$, we say that $j\in[p]$ is a \emph{parent} of $i$ in $\GG$ if $j\rightarrow i\in E$, we say $j$ is a \emph{spouse} of $i$ if $i\leftrightarrow j\in E$, we say $j$ is a \emph{child} of $i$ if $i\rightarrow j\in E$, and we let $\pa_\GG(i)$, $\sp_\GG(i)$, and $\de_\GG(i)$ denote the parents, spouses, and children of $i$ in $\GG$, respectively.  
A \emph{path} in $\GG$ is a sequence $\pi = \langle v_1,e_1,v_2,e_2,\ldots,v_{m-1},e_{m-1},v_m\rangle$ where $v_k\in [p]$ and $e_k \in E$ for all $k$ and $e_k$ has endpoints $v_k$ and $v_{k-1}$.  
We call $v_1$ and $v_m$ the \emph{endpoints} of $\pi$.  
If $A\subset[p]$ such that for each $a,b\in A$ there exists some path $\pi$ in $\GG$ between $a$ and $b$ using only vertices in $A$, then $A$ (or the induced subgraph $\GG\langle A\rangle$) is called \emph{path-connected} in $\GG$.  
For a path $\pi$, if $e\in E$ is an edge with endpoints $v_1$ and $v_m$ such that $e$ has an arrowhead at the endpoint $e_1$ if and only if $e_1$ has an arrowhead at $v_1$, and similarly for $e_{m-1}$ and $v_m$, then we call $e$ an \emph{endpoint-identical edge} for $\pi$.  
Any subsequence of $\pi$ that is also a path is called a \emph{subpath} of $\pi$.  
A path $\pi$ is called \emph{directed} if $e_k = v_k\rightarrow v_{k+1}$ for all $k$, and it is called \emph{anterior} if there exists $t\in [m]$ such that $e_k = v_k - v_{k+1}$ for all $k \leq t$ and $e_k = v_k \rightarrow v_{k+1}$ for all $k > t$.
A \emph{cycle} is a sequence $\pi = \langle v_1,e_1,v_2,e_2,\ldots,v_{m-1},e_{m-1},v_m\rangle$ where $v_k\in [p]$ and $e_k \in E$ for all $k$, $e_k$ has endpoints $v_k$ and $v_{k-1}$, and $v_m = v_1$.  
It is called \emph{directed} if $e_k = v_k\rightarrow v_{k+1}$ for all $k$. 
When the edges in a path or cycle $\pi$ are understood, we will simply write $\pi$ as the sequence of vertices $\pi = \langle v_1,v_2,\ldots,v_m\rangle$.
A node $j\in[p]$ is an \emph{ancestor} of $i\in[p]$ in $\GG$ if there exists a directed path $\pi$ in $\GG$ with $v_1 = j$ and $v_m = i$. 
In this case $i$ is called a \emph{descendant} of $j$ in $\GG$.  
Similarly, $j$ is \emph{anterior} to $i$ if the path $\pi$ is an anterior path.  
We let $\an_\GG(i)$, $\de_\GG(i)$, and $\ant_\GG(i)$ denote the collection of nodes in $\GG$ that are ancestors, descendants, or anterior to $i$, respectively.

A vertex $v$ on a path $\pi = \langle v_1,v_2,\ldots,v_m\rangle$ that is not an endpoint of $\pi$ is called a \emph{collider} if two arrowheads point to $v$ on $\pi$.  
A \emph{collider path} is a path on which every vertex is a collider except for its endpoints.  
A collider path on three nodes $\langle v_1,v_2,v_3\rangle$ is called a \emph{v-structure} (or an \emph{unshielded collider} or an \emph{immorality}) \cite{ARS09,GP01} if $v_1$ and $v_3$ are not adjacent in $\GG$.  

An LMG $\GG$ is called \emph{maximal} if adding any edge to $\GG$ changes the collection $\J(\GG)$ of m-separation statements.  
To identify the unique maximal graph containing an ancestral graph (AG), we consider its inducing paths: an \emph{inducing path} $\pi$ between vertices $i$ and $j$ of an AG $\GG$ is a path on which any non-endpoint vertex is both a collider on $\pi$ and an ancestor of at least one of $i$ or $j$ in $\GG$.  
When discussing AGs, we will also utilize the notion of discriminating paths, whose existence in a pair of MAGs implies that a certain collider path $\langle i,j,k\rangle$ will be the same in both graphs, even though $i$ and $k$ are adjacent.  
A path $\pi = \langle v_0,v_1,\ldots,v_m,j,k\rangle$ is a \emph{discriminating path} for $\langle v_m,j,k\rangle$ in a MAG $\GG$ if $v_0$ is not adjacent to $k$, and for all $i\in[m]$, the vertex $v_i$ is a collider on $\pi$ and a parent of $k$.  
Using discriminating paths, we can also define \emph{triples with order}.  
A triple $\langle i,j,k\rangle$ has \emph{order $0$} if $i$ and $k$ are not adjacent.  
For $t\geq0$, if we let $\mathfrak{D}_t$ denote the set of all \emph{triples of order $t$} in a given MAG $\GG$, then $\langle i,j,k\rangle\in\mathfrak{D}_{t+1}$ if it is not a triple with order $s<t+1$, and there is a discriminating path $\langle v_0,v_1,\ldots,v_m,b,c\rangle$ with either $\langle i,j,k\rangle = \langle v_m,b,c\rangle$ or $\langle k,j,i\rangle = \langle v_m,b,c\rangle$ such that 
\[
\langle v_0,v_1,v_2\rangle, \langle v_1,v_2,v_3\rangle, \ldots, \langle v_{m-1},v_m,b\rangle\in\bigcup_{s\leq t}\mathfrak{D}_s.
\]
A graph $\GG$ is called \emph{anterior} if $\GG$ contains no triples of the form $i \qrightarrow j - k$.  
Given an LMG $\GG$, the \emph{anterior graph of $\GG$}, denoted $\GG^\ast$, is the graph produced by consecutively removing all arrowheads pointing into nodes adjacent to an undirected edge; i.e. $i \leftrightarrow j - k$ becomes $i \leftarrow j - k$ and $i \rightarrow j - k$ becomes $i - j - k$. 
We will use the following lemma from \cite{SL14} on the anterior graphs of ribbonless graphs (RGs):  
\begin{lemma}\cite[Proposition 1]{SL14}
\label{lem: RG anterior graph}
If $\GG$ is an RG then 
$
\J(\GG) = \J(\GG^\ast).
$
\end{lemma}

We will also require a few lemmas about maximal graphs, the first of which is immediate:

\begin{lemma}
\label{lem: maximal graph}
An LMG $\GG = ([p],E)$ is maximal if and only if for every pair of nonadjacent nodes $i,j\in[p]$ there exists a subset $C\subset[p]\backslash\{i,j\}$ such that 
$
\langle i, j\mid C\rangle \in \J(\GG).
$
\end{lemma}

The next lemma takes a bit more work, but is a natural extension of an observation made by \cite{ARS09}.
For the sake of completeness, we give a proof here.
  
\begin{lemma}
\label{lem: maximal graphs, adjacencies, v-structures}
Let $\GG$ and $\HH$ be Markov equivalent maximal simple LMGs on node set $[p]$.
Then $\GG$ and $\HH$ have the same adjacencies and v-structures.
\end{lemma}

\begin{proof}
Suppose first that $i$ and $j$ are adjacent in $\GG$ but not in $\HH$.  
Since they are not adjacent in $\HH$ then, by \cite[Lemma 4]{SL14}, there exists a subset $C$ of $[p]\setminus\{i,j\}$ such that $i$ is $m$-separated from $j$ given $C$ in $\HH$.
Therefore, since $\J(\GG) = \J(\HH)$ and $\langle i,j \mid C\rangle\in \J(\HH)$, it follows that $i$ is $m$-separated from $j$ given $C$ in $\GG$.  
However, this contradicts the fact that $i$ and $j$ are adjacent, since an adjacency is always an $m$-connecting path.  
By symmetry of this argument, we deduce that $\GG$ and $\HH$ have the same adjacencies.  

Suppose now that $i \qrightarrow j\qleftarrow k$ forms a v-structure in $\GG$ but not in $\HH$.  
Since $i$ and $k$ are not adjacent, then by maximality of $\GG$ we know that there exists a triple $\langle i, k \mid C\rangle\in \J(\GG)$ for some $C\subset[p]\setminus\{i,j\}$.  
By Markov equivalence of $\GG$ and $\HH$, we know that $\langle i, k, \mid C\rangle\in\J(\HH)$ as well. 
Moreover, by the previous argument, we know that $i$ and $k$ are not adjacent in $\HH$.  
Note now that we must have $j\notin C$, since otherwise $i \qrightarrow j\qleftarrow k$ would form an $m$-connecting path given $C$ in $\GG$. 
Thus, since $j\notin C$, $i$ and $k$ are not adjacent in $\HH$, and the tripath $\langle i,j,k\rangle$ in $\HH$ is not a collider path (since it is not a v-structure), then the path $\langle i,j,k\rangle$ must be $m$-connecting given $C$ in $\HH$.  
However, this contradicts the fact that $\langle i, k \mid C\rangle\in \J(\HH)$.  
Thus, we conclude that $\GG$ and $\HH$ must have the same v-structures.  
\end{proof}

	\begin{figure}
	\centering

\begin{tikzpicture}[thick,scale=.5]

	 \node[draw,align=center, color = black!100] (ZAP70) at (-9,0) {\scriptsize ZAP70};
	 \node[draw,align=center, color = black!100] (LCK) at (-6,0) {\scriptsize LCK};
	 \node[draw,align=center, color = black!100] (SLP-76) at (-3,0) {\scriptsize SLP-76};
	 \node[draw,align=center, color = black!100] (PI3K) at (0,0) {\scriptsize PI3K};
	 \node[draw,align=center, color = black!100] (RAS) at (3,0) {\scriptsize RAS};
	 \node[draw,align=center, color = black!100] (Cytohesin) at (6,0) {\scriptsize Cytohesin};
	 \node[draw,align=center, color = black!100] (PKC) at (2,-2) {\scriptsize PKC};
	 \node[draw,align=center, color = black!100] (PLC) at (-3,-3) {\scriptsize PLC$\gamma$};
	 \node[draw,align=center, color = black!100] (PIP3) at (-.5,-4.5) {\scriptsize PIP3};
	 \node[draw,align=center, color = black!100] (PIP2) at (-4,-5) {\scriptsize PIP2};
	 \node[draw,align=center, color = black!100] (Akt) at (2,-4) {\scriptsize Akt};
	 \node[draw,align=center, color = black!100] (PKA) at (2.5,-6) {\scriptsize PKA};
	 \node[draw,align=center, color = black!100] (Raf) at (4.5,-6) {\scriptsize Raf};
	 \node[draw,align=center, color = black!100] (Mek$1/2$) at (7.5,-6) {\scriptsize Mek$1/2$};
	 \node[draw,align=center, color = black!100] (Erk$1/2$) at (6,-8) {\scriptsize Erk$1/2$};
	 \node[draw,align=center, color = black!100] (P38) at (2.5,-8) {\scriptsize P38};
	 \node[draw,align=center, color = black!100] (JNK) at (0,-6.5) {\scriptsize JNK};
	 	
	 \draw[->]   (ZAP70) -- (LCK) ;
 	 \draw[->]   (LCK) -- (SLP-76) ;
 	 \draw[->]   (SLP-76) -- (PI3K) ;
 	 \draw[->]   (PI3K) -- (RAS) ;
 	 \draw[->]   (Cytohesin) -- (RAS) ;
 	 \draw[->]   (PKC) -- (RAS) ;
 	 \draw[->]   (PKC) -- (PIP2) ;
 	 \draw[->]   (PLC) -- (PKC) ;
 	 \draw[->]   (PIP3) -- (PLC) ;
 	 \draw[<->]   (PIP2) -- (PIP3) ;
 	 \draw  (PIP2) edge[<->,bend left=45] (PI3K) ;
 	 \draw[->]   (PLC) -- (PIP2) ;
 	 \draw[->]   (PIP3) -- (Akt) ;
 	 \draw[->]   (PKA) -- (Akt) ;
 	 \draw[->]   (PKC) -- (Raf) ;
 	 \draw[->]   (RAS) -- (Raf) ;
 	 \draw[->]   (Raf) -- (Mek$1/2$) ;
 	 \draw[->]   (Mek$1/2$) -- (Erk$1/2$) ;
 	 \draw[->]   (PKA) -- (P38) ;
 	 \draw[->]   (PKC) -- (JNK) ;
	 
	 \node[circle, draw, color = ist green!90, fill=black!0, inner sep=1pt, minimum width=1pt] (1) at (-10,-1.5) {\scriptsize$1$};
	 \node[circle, draw, color = ist green!90, fill=black!0, inner sep=1pt, minimum width=1pt] (2) at (1,1.25) {\scriptsize$2$};
	 \node[circle, draw, color = ist green!90, fill=black!0, inner sep=1pt, minimum width=1pt] (3) at (8,-1.5) {\scriptsize$3$};
	 \node[circle, draw, color = ist green!90, fill=black!0, inner sep=1pt, minimum width=1pt] (4) at (1.7,-.5) {\scriptsize$4$};
	 \node[circle, draw, color = ist green!90, fill=black!0, inner sep=1pt, minimum width=1pt] (5) at (1.5,-7) {\scriptsize$5$};
	 \node[circle, draw, color = red!90, fill=black!0, inner sep=1pt, minimum width=1pt] (6) at (1,-.9) {\scriptsize$6$};
	 \node[circle, draw, color = red!90, fill=black!0, inner sep=1pt, minimum width=1pt] (7) at (1.25,-5) {\scriptsize$7$};
	 \node[circle, draw, color = red!90, fill=black!0, inner sep=1pt, minimum width=1pt] (8) at (-5,-3.5) {\scriptsize$8$};
	 \node[circle, draw, color = red!90, fill=black!0, inner sep=1pt, minimum width=1pt] (9) at (7,-4.5) {\scriptsize$9$};
	 \node[circle, draw, color = red!90, fill=black!0, inner sep=1pt, minimum width=1pt] (10) at (-1,1) {\scriptsize$10$};

	  \draw[->,ist green]   (1) -- (ZAP70) ;
	  \draw[->,ist green]   (2) -- (PI3K) ;
	  \draw[->,ist green]   (3) -- (Cytohesin) ;
	  \draw[->,ist green]   (4) -- (PKC) ;
	  \draw[->,ist green]   (5) -- (PKA) ;
	  \draw[->,red]   (6) -- (PKC) ;
	  \draw[->,red]   (7) -- (Akt) ;
	  \draw[->,red]   (8) -- (PIP2) ;
	  \draw[->,red]   (9) -- (Mek$1/2$) ;
	  \draw[->,red]   (10) -- (PI3K) ;
	  
	 \node (key1) at (-8.5,-5) {\scriptsize \textcolor{ist green}{activating intervention}} ;
	 \node (key2) at (-8.5,-5.75) {\scriptsize \textcolor{red}{inhibiting intervention}} ;

\end{tikzpicture}
	\vspace{-0.2cm}
	\caption{An accepted ground-truth protein signaling network studied in primary human immune system cells \protect\cite{Sachs05}. Red and green nodes represent the points of intervention, with red nodes denoting inhibitory interventions and green nodes denoting activating interventions.  When all nodes and arrows, black, red, and green, are considered, this is an example of an $\I$-LMG that is not an $\I$-DAG.  As this is the accepted ground-truth in the biology, researchers need to distinguish it from other Markov equivalent LMGs based on their interventional experiments (see for instance \protect\cite{Sachs05,HB15}).  Since it is not a DAG we cannot apply the graphical characterization of \protect\cite{YKU18} to do so.  Hence, a more general theory of $\I$-Markov equivalence is needed.  Developing such a theory is the focus of this paper. }
	\label{fig: Sachs}
	\end{figure}
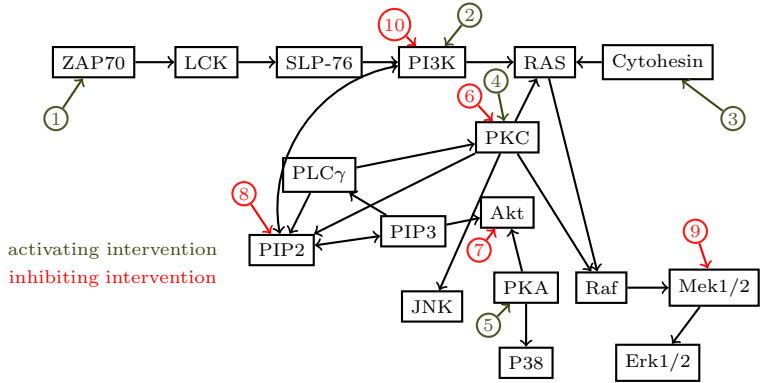
	
\section{Lemmata and Proofs for Section~\ref{sec: I-Markov equivalence for MAGs}} 
\label{app: lemmata and proofs for sec 3}
We will require the following fact about the $\I$-graph produced by intervening on an AG without doubly-intervening on any selection adjacent nodes.
\begin{lemma}
\label{lem: ribbonless}
Let $\GG = ([p],E)$ be an AG, and $\I$ a collection of intervention targets.  
Then $\I$ does not doubly-intervene on any selection adjacent nodes of $\GG$ if and only if $\GG^\I$ is ribbonless.  
\end{lemma}

\begin{proof}
Suppose that $\GG^\I$ contains no doubly-intervened selection adjacent nodes.  
If $\GG^\I$ contains a ribbon then there exists a collider path $\langle i,j,k\rangle$ such that (1) $i$ and $k$ have no endpoint identical edge between them, and (2) $j$ or a descendant of $j$ is an endpoint of a line, or $j$ is in a directed cycle.
Since $\GG$ is an AG, it contains no directed cycle. 
So any ribbon $\langle i,j,k\rangle$ in $\GG^\I$ must have $j$ or a descendant of $j$ as the endpoint of an undirected edge.  
Since $\GG$ is an AG, this could only happen if $i,k\in W_\I$, meaning that $j$ must be a doubly-intervened selection adjacent node.  
Hence, $\GG^\I$ must be ribbonless.  

Conversely, suppose that $\GG^\I$ is ribbonless.  
Let $j$ be a selection adjacent node in $\GG$.  
If $\omega_I,\omega_J\in W_\I$ are such that $j\in I\cap J$, then the collider path $\langle \omega_I,j,\omega_J\rangle$ would be a ribbon in $\GG^\I$.  
Hence, $\GG^\I$ contains no doubly-intervened selection adjacent nodes. 
\end{proof}

To generalize Theorem~\ref{thm: YKU characterization} to AGs, we will work with maximal AGs, or MAGs.  
Our first goal is to show that this is sufficient in the sense that two AGs $\GG$ and $\HH$ are $\I$-Markov equivalent if and only if their unique maximal extensions $\overline{\GG}$ and $\overline{\HH}$ are $\I$-Markov equivalent.  
To do so, we need the following two lemmas. 
Recall that, given $\GG = ([p], E)$ and a subset $A\subset[p]$, we let $\GG\langle A\rangle$ denote the induced subgraph of $\GG$ on $A$.

\begin{lemma}
\label{lem: restriction}
Let $\GG = ([p],E)$ be an AG, and $\I$ a collection of intervention targets that does not doubly-intervene on any selection adjacent nodes of $\GG$.  
Then 
$
\overline{\GG} = \overline{\GG^\I}\langle [p]\rangle.
$
\end{lemma}

\begin{proof}
Note first, by Lemma~\ref{lem: ribbonless}, $\overline{\GG^\I}$ is well-defined since $\GG^\I$ is ribbonless.  
We now show the inclusion $\overline{\GG}\supseteq\overline{\GG^\I}\langle[p]\rangle$.  
First, recall from \cite[Theorem 5.1]{RS02} that, in an AG $\GG$, if two nodes $i$ and $j$ are adjacent in $\overline{\GG}$ but not in $\GG$ then they are connected by a bidirected edge in $\overline{\GG}$ and an inducing path containing at least two edges in $\GG$.  
Hence, to prove the desired inclusion, we must show that if $i,j\in[p]$ are connected by a bidirected edge in $\overline{\GG^\I}$ but are not adjacent in $\GG^\I$, then $i$ and $j$ are connected by a bidirected edge in $\overline{\GG}$.  
To this end, note that if $i,j\in[p]$ are adjacent in $\overline{\GG^\I}$ but not in $\GG^\I$ then they are connected by an inducing path in $\GG^\I$ containing at least two edges, say $\pi = \langle v_0 = i, v_1,\ldots,v_{m-1},v_m = j\rangle$ where $m\geq2$.   
Thus, to prove that $i$ and $j$ are also connected by a bidirected edge in $\overline{\GG}$, it suffices to show that $\pi$ is an inducing path in $\GG$. 
However, by \cite[Lemma 4.5 (iii)]{RS02}, we know that every edge on $\pi$ is bidirected.  
Hence, since each node in $W_\I$ is a source node, we know that $v_1,\ldots, v_{m-1}\in[p]$ and $\pi$ is a path in $\GG$.  
Similarly, since $\pi$ is an inducing path in $\GG^\I$ then each $v_k$, $k = 1,\ldots, m-1$, is an ancestor of $i$ or $j$.  
So, without loss of generality, there exists a directed path $s = \langle v_0^\prime = v_k,v_1^\prime, \ldots, v_t^\prime = i\rangle$ from $v_k$ to $i$ in $\GG^\I$.  
However, since all nodes in $W_\I$ are source nodes in $\GG^\I$, then no vertex or edge on $s$ can be in $W_\I$ or $E_\I$, respectively.  
Hence, $s$ is also a directed path in $\GG$. 
It then follows that $\pi$ is an inducing path in $\GG$, and we conclude that $\overline{\GG}\supseteq\overline{\GG^\I}\langle[p]\rangle$.

To see the reverse inclusion, we must show that if $i,j\in[p]$ are nonadjacent in $\GG$ but adjacent in $\overline{\GG}$ (and therefore connected by a bidirected edge), then $i$ and $j$ are connected by a bidirected edge in $\overline{\GG^\I}\langle[p]\rangle$ as well.  
So suppose that $i,j\in[p]$ are nonadjacent in $\GG$ but adjacent in $\overline{\GG}$.  
For the sake of contradiction, suppose now that $i$ and $j$ are not connected by a bidirected edge in $\overline{\GG^\I}\langle[p]\rangle$.  
Then they are nonadjacent in $\overline{\GG^\I}\langle[p]\rangle$.  
Since $i$ and $j$ are adjacent in $\overline{\GG}$, then there exists an inducing path containing at least two edges in $\GG$ between $i$ and $j$.  
Then the same inducing path must also be in $\GG^\I$ and in $\overline{\GG^\I}\langle[p]\rangle$.  
However, by \cite[Corollary 4.6]{RS02}, this contradicts the maximality of $\overline{\GG^\I}$.  
Thus, we conclude that $\overline{\GG^\I}\langle[p]\rangle\supseteq\overline{\GG}$, and hence $\overline{\GG^\I}\langle[p]\rangle=\overline{\GG}$. 
\end{proof}

\begin{lemma}
\label{lem: no additional intervening}
Let $\GG = ([p],E)$ be an AG, and $\I$ a collection of intervention targets that does not doubly-intervene on any selection adjacent nodes of $\GG$.  
If $i\in W_\I$ and $j\in[p]$ such that $i$ and $j$ are not adjacent in $\GG^\I$, then $i$ and $j$ are not adjacent in $\overline{\GG^\I}$.
\end{lemma}

\begin{proof}
By Lemma~\ref{lem: maximal graph} it suffices to prove that 
\[
i\independent j\mid\left(\ant(j)\cup W_\I\right)\backslash\{i,j\}.
\]
Suppose on the contrary, that there exists an $m$-connecting path $\pi = \langle i,q_1,\ldots,q_m, j\rangle$ given $C= \left(\ant(j)\cup W_\I\right)\backslash\{i,j\}$ in $\GG^\I$.  
Notice first that since $i\in W_\I$, we know that $i\rightarrow q_1$.  
Moreover, since all nodes in $W_\I$ are source nodes in $\GG^\I$ and $W_\I\backslash\{i\}\subseteq C$, then $q_1,\ldots,q_m\in[p]$; that is, the path $\langle q_1,\ldots,q_m\rangle$ is contained in $\GG$.  
Since $\ant(j)\backslash\{j\}\subseteq C$, we know that $\pi$ is not of the form
\[
i\rightarrow q_1 - q_2 - \cdots - q_{k-1} \rightarrow q_k \rightarrow q_{k+1} \rightarrow \cdots \rightarrow q_m \rightarrow j.
\]
Hence, there must be some collider on $\pi$.  
Take the collider path $\langle q_{t-1},q_t,q_{t+1}\rangle$ on $\pi$ with $t$ maximum.  
Since $\pi$ is $m$-connecting then $q_t\in C\cup\an(C)$.  
Since every node in $W_\I$ is a source node, it must be that $q_t\in\ant(j)\backslash\{j\}$.  
Since $t$ is maximum, then $q_{t+1},\ldots,q_m$ are noncolliders on $\pi$, and hence not in $\ant(j)\setminus\{j\}$.  
Hence, the subpath $\langle q_{t+1},\ldots,q_m,j\rangle$ of $\pi$ must be of the form 
\[
q_{t+1}\leftarrow q_{t+2}\leftarrow \cdots \leftarrow j.
\]
It follows that $q_t\in\an(j)\setminus\{j\}$.  
This is because the node $q_t$ could only be in $\ant(j)\setminus(\an(j)\cup\{j\})$ if $q_{t+1}\in W_\I$, but this is impossible since we have already seen that $q_1,\ldots,q_m,j\in[p]$.  
As well, it also follows that the edge between $q_t$ and $q_{t+1}$ cannot be bidirected, as this would contradict the maximality of $t$.  
Therefore, the edge between $q_t$ and $q_{t+1}$ must be of the form $q_t\leftarrow q_{t+1}$.  
However, since $q_t\in\an(j)\setminus\{j\}$, it then follows that there is a directed cycle in $\GG$, which contradicts the assumption that $\GG$ is ancestral.  
Thus, we conclude that no such $m$-connecting path exists, completing the proof.
\end{proof}

Using Lemmas~\ref{lem: restriction} and ~\ref{lem: no additional intervening}, we get the following:

\begin{proposition}
\label{prop: commuting}
Let $\GG = ([p],E)$ be an AG, and $\I$ a collection of intervention targets that does not doubly-intervene on any selection adjacent nodes in $\GG$.  Then
$
\overline{\GG}^\I = \overline{\GG^\I}.
$
\end{proposition}

\begin{proof}
By Lemmas~\ref{lem: restriction} and~\ref{lem: no additional intervening}, we know that the only possible adjacencies in $\overline{\GG^\I}$ that may not be in $\overline{\GG}^\I$ would be between nodes in $W_\I$.  
By Lemma~\ref{lem: ribbonless}, we know $\GG^\I$ is ribbonless.  
Hence, by \cite[Algorithm 2.2]{S12}, we know that $\overline{\GG^\I}$ is produced from $\GG^\I$ by adding endpoint identical edges between the endpoints of inducing paths in $\GG^\I$.  
Since every node in $W_\I$ is a source node in $\GG^\I$, we also know there are no arrowheads pointing to these nodes.  
Thus, by \cite[Proposition 2.6]{S12}, we know that there is no inducing path between nodes in $W_\I$.  
Thus, we conclude that $\overline{\GG}^\I = \overline{\GG^\I}.$
\end{proof}

Proposition~\ref{prop: commuting} tells us that if $\I$ is a collection of intervention targets that does not doubly-intervene on any selection adjacent nodes in either of two AGs $\GG$ and $\HH$, then $\GG$ and $\HH$ are $\I$-Markov equivalent if and only if $\overline{\GG}$ and $\overline{\HH}$ are $\I$-Markov equivalent.  
This is because we have
\begin{equation*}
\begin{split}
\overline{\HH}^\I = \overline{\HH^\I}\approx \HH^\I &\approx \GG^\I \approx \overline{\GG^\I} = \overline{\GG}^\I, \mbox{ and} \\ 
{\HH}^\I \approx \overline{\HH^\I} =  \overline{\HH}^\I &\approx \overline{\GG}^\I = \overline{\GG^\I} \approx {\GG}^\I.\\
\end{split}
\end{equation*}
Hence, it suffices to characterize $\I$-Markov equivalence when $\GG$ and $\HH$ are assumed to be MAGs.
Another immediate consequence of Proposition~\ref{prop: commuting} is the following:

\begin{lemma}
\label{lem: anterior graphs of interventions}
Suppose $\GG = ([p],E)$ is an AG and $\I$ a collection of intervention targets that does not doubly-intervene on any nodes of $\GG$. Then $\GG$ is maximal if and only if $(\GG^\I)^\ast$ is a MAG that is Markov equivalent to $\GG^\I$.
\end{lemma}

\begin{proof}
First assume that $\GG$ is maximal; i.e., that $\GG$ is a MAG.
By Lemma~\ref{lem: ribbonless}, we know that $\GG^\I$ is ribbonless.  
By \cite[Proposition 1]{SL14}, we know that $\GG^\I$ is Markov equivalent to $(\GG^\I)^\ast$.  
To see that $(\GG^\I)^\ast$ is a MAG, notice by Proposition~\ref{prop: commuting}, $\GG^\I$ is maximal since
$\GG = \overline{\GG}$ implies $\GG^\I = \overline{\GG}^\I = \overline{\GG^\I}$.
It follows that $(\GG^\I)^\ast$ is also maximal, since removing arrowheads pointing into selection adjacent nodes cannot generate new inducing paths.  
Moreover, since $\GG$ is a MAG, then $\GG^\I$ contains no directed cycles nor bidirected edges with a directed path between its endpoints.  
Hence, neither does $(\GG^\I)^\ast$.  
Thus, $(\GG^\I)^\ast$ is a MAG. 

Conversely, suppose that $(\GG^\I)^\ast$ is a MAG that is Markov equivalent to $\GG^\I$.
Since $\GG$ is ancestral and $\I$ does not doubly-intervene on any selection adjacent nodes in $\GG$, we know that $\GG^\I$ is produced from $(\GG^\I)^\ast$ by replacing undirected edges of the form $\omega_I - j$ for some $\omega_I\in W_\I$ and $j\in\sa(\GG)$ a selection adjacent node, with directed arrows $\omega_I \rightarrow j$.  
Since $\GG$ is ancestral and $\I$ does not doubly-intervene on any selection adjacent nodes, doing so cannot produce any new colliders, and hence $\GG^\I$ does not containing any inducing paths that are not in $(\GG^\I)^\ast$.  Since, by Lemma~\ref{lem: ribbonless}, $\GG^\I$ is ribbonless, it follows from \cite[Theorem 2]{SL14} that $\GG^\I$ is maximal; i.e., $\overline{\GG^\I} = \GG^\I$. 
So by Proposition~\ref{prop: commuting}, 
$
\GG^\I = \overline{\GG^\I} = \overline{\GG}^\I.
$
Thus, $\GG = \overline{\GG}$; i.e., $\GG$ is maximal.
\end{proof}

\subsection{Proof of Theorem~\ref{thm: main}} 

In order to prove Theorem~\ref{thm: main}, we will use the following lemma.
\begin{lemma}
\label{lem: I-colliders with order}
Let $\langle i,j,k\rangle$ be an $\I$-collider with order $t$ in a LMG $\GG$ for a collection of interventions $\I$.  
Then there is a discriminating path 
$
\pi = \langle v_0,\ldots, v_m,b,c\rangle
$
in $\GG^\I$ for which
\begin{enumerate}
	\item  $\langle i,j,k\rangle \in\{\langle v_m,b,c\rangle ,\langle c,b,v_m\rangle \}$,
	\item  the subpaths $\langle v_0,v_1,v_2\rangle, \langle v_1,v_2,v_3\rangle, \ldots, \langle v_{m-1},v_m,b\rangle $ are all colliders with order less than or equal to $t$, and
	\item $\langle v_0,v_1,v_2\rangle$ is an $\I$-v-structure in $\GG^\I$.  
\end{enumerate}
\end{lemma}

\begin{proof}
Let $\langle i,j,k\rangle$ be an $\I$-collider with order $t$ in $\GG$.  
Then $\langle i,j,k\rangle$ is not a collider with order in $\GG$.  
So, by \cite[Definition 3.11]{ARS09}, there does not exist a discriminating path $\pi = \langle v_0,\ldots, v_m,b,c\rangle$ in $\GG$ satisfying properties (1) and (2).  
However, since $\langle i,j,k \rangle$ is a collider with order in $\GG^\I$, it follows that there exists such a discriminating path $\pi$ in $\GG^\I$.  
Therefore, it must be that at least one of the vertices $v_0,\ldots,v_m,b,c$ is an element of $W_\I$.  
Supposing that $m\neq 0$, it follows that $v_1,\ldots,v_m\notin W_\I$ since $W_\I$ contains only source nodes.  
Similarly, since (without loss of generality) $\langle v_m,b,c\rangle = \langle i,j,k\rangle$ , which is a collider, it follows that at least one of $v_0$ and $c$ is in $W_\I$.  
Since $\pi$ is a discriminating path, it follows that $v_1,\ldots, v_m$ are all parents of $c$ in the case that $m\neq0$.  
Since all nodes in $W_\I$ are source nodes, it follows that $c\in W_\I$ if and only if $m=0$ and $\langle i,j,k\rangle$ is an $\I$-v-structure.  
In this case, we have that $\langle v_0,b,c\rangle = \langle i,j,k\rangle$, as desired.  
On the other hand, if $c\notin W_\I$, it must be that $v_0\in W_\I$.  
By \cite[Definition 3.11]{ARS09}, it follows that $\langle v_0,v_1,v_2\rangle$ is a collider with order at most $t$.  
If $\langle v_0,v_1,v_2\rangle$ is a collider with order $0$, it is an $\I$-v-structure and we are done.  

So suppose, for the sake of contradiction, that $\langle v_0,v_1,v_2\rangle$ has order $i>0$.  
Then, by \cite[Definition 3.11]{ARS09}, $\langle v_0,v_1,v_2\rangle$ is not a v-structure, and there exists a discriminating path $\langle q_0,\ldots, q_n,x,y\rangle$ in $\GG^\I$ for which $\langle v_0,v_1,v_2\rangle\in\{ \langle q_n,x,y\rangle, \langle y,x,q_n\rangle\}$.  
By definition of a discriminating path, it follows that all of $q_1,\ldots, q_n$ are parents of $y$.  
Since, $\langle v_0,v_1,v_2\rangle$ is not a v-structure, it must be that $m\neq 0$, and thus, $y$ has at least one parent.  
Hence, $y\notin W_\I$, as $W_\I$ contains only source nodes.  
Similarly, since $m\neq0$, then, by \cite[Definition 3.11]{ARS09}, it must be that $\langle q_{n-1},q_{n}, x\rangle$ is a collider in $\GG^\I$.  
However, this implies that $q_n$ is not a source node, and therefore $q_n\notin W_\I$.  
This yields a contradiction to the assumption that $i>0$.  
Thus, we conclude that $\langle v_0,v_1,v_2\rangle$ is a collider with order $0$, and therefore it is an $\I$-v-structure in $\GG^\I$.  
\end{proof}

\subsubsection{ Proof of Theorem~\ref{thm: main}}
The equivalence of (1) and (2) is an immediate consequence of Proposition~\ref{prop: commuting}.
So it only remains to show (2) if and only if (3). 
For notational simplicity, in the remainder of this proof we assume $\GG$ and $\HH$ are MAGs.
Suppose first that $\GG$ and $\HH$ are $\I$-Markov equivalent.  
By Lemma~\ref{lem: ribbonless}, we know that $\GG^\I$ and $\HH^\I$ are ribbonless graphs satisfying $\J(\GG^\I) = \J(\HH^\I)$.
Hence, $\J(\GG) = \J(\HH)$, and so it only remains to show that $\GG^\I$ and $\HH^\I$ have the same $\I$-colliders with order.  

Since $\GG$ and $\HH$ are maximal, by Lemma~\ref{lem: anterior graphs of interventions}, we know that we have that 
\[
\J((\GG^\I)^\ast) = \J(\GG^\I) = \J(\HH^\I) = \J((\HH^\I)^\ast),
\]
and $(\GG^\I)^\ast$ and $(\HH^\I)^\ast$ are MAGs.  
Therefore, $(\GG^\I)^\ast$ and $(\HH^\I)^\ast$ are Markov equivalent MAGs, and by \cite[Theorem 3.7]{ARS09}, they have the same colliders with order.
Suppose now that $\langle i,j,k\rangle$ is an $\I$-collider with order $t$ in $\GG^\I$.  
Then, by Lemma~\ref{lem: I-colliders with order} there is a discriminating path $\pi = \langle v_0,\ldots, v_m,b,c\rangle$ in $\GG^\I$ for which $\langle i,j,k\rangle \in\{\langle v_m,b,c\rangle ,\langle c,b,v_m\rangle \}$,  the subpaths $\langle v_0,v_1,v_2\rangle, \langle v_1,v_2,v_3\rangle, \ldots, \langle v_{m-1},v_m,b\rangle $ are all colliders with order less than or equal to $t$, and $\langle v_0,v_1,v_2\rangle$ is an $\I$-v-structure in $\GG^\I$. 
It follows that $v_0 = \omega_I$ for some $I\in\I$ and $v_1\in I$.  
Moreover, since $\langle v_0,v_1,v_2\rangle$ is a v-structure, the edge in $\GG^\I$ between $v_1$ and $v_2$ is of the form $v_1\qleftarrow v_2$.  
Hence, since $\GG$ is an ancestral graph, it must be that there are no arrows of the form $v_1\qleftarrow v$ in $\GG^\I$.  
Therefore, the arrow $v_0\rightarrow v_1$ is also an arrow in $(\GG^\I)^\ast$.  
Since the only difference between $(\GG^\I)^\ast$ and $\GG^\I$ is that some arrows from nodes in $W_\I$ are undirected as opposed to directed, it follows that $(\GG^\I)^\ast$ and $\GG^\I$ have identical discriminating paths.  

It follows that if $\langle i,j,k\rangle$ is an $\I$-collider with order $t$ in $\GG^\I$, then it is a collider with order in $(\GG^\I)^\ast$.  
Since $(\GG^\I)^\ast$ and $(\HH^\I)^\ast$ are Markov equivalent MAGs, it follows from \cite[Theorem 3.7]{ARS09} that $\langle i,j,k\rangle$ is also a collider with order in $(\HH^\I)^\ast$.  
Since, as we argued above, $\HH^\I$ and $(\HH^\I)^\ast$ have the same discriminating paths for colliders with order, we know that $\langle i,j,k\rangle$ is a collider with order in $\HH^\I$.  
Therefore, it only remains to see that $\langle i,j,k\rangle$ is further an $\I$-collider with order in $\HH^\I$.  

Suppose, on the contrary, that $\langle i,j,k\rangle$ is not an $\I$-collider with order in $\HH^\I$. 
Then $\langle i,j,k\rangle$ is a collider with order determined by a discriminating path that does not use any nodes in $W_\I$.  
Since $\GG$ and $\HH$ are Markov equivalent MAGs, it must be that $\langle i,j,k\rangle$ is also a collider with order in $\GG$.  
However, this contradicts the assumption that $\langle i,j,k\rangle$ is an $\I$-collider with order in $\GG^\I$.  
Hence, $\langle i,j,k\rangle$ must be a collider with order in $\HH^\I$ determined only by discriminating paths that use at least one node in $W_\I$.  
Hence, $\langle i,j,k\rangle$ is an $\I$-collider in $\HH^\I$.

To see the converse, suppose now that $\GG$ and $\HH$ are Markov equivalent MAGs and that $\GG^\I$ and $\HH^\I$ have the same $\I$-colliders with order. 
Since $\GG$ and $\HH$ are Markov equivalent MAGs, and since $\I$ does not doubly-intervene at any selection adjacent nodes in either of $\GG$ or $\HH$, then by Lemma~\ref{lem: anterior graphs of interventions}, we know that 
\[
\J((\GG^\I)^\ast) = \J(\GG^\I),
\qquad
\J((\HH^\I)^\ast) = \J(\HH^\I),
\]
and that $(\GG^\I)^\ast$ and $(\HH^\I)^\ast$ are MAGs.
Hence, it suffices to show that $\J((\GG^\I)^\ast) = \J((\HH^\I)^\ast)$.
Since $(\GG^\I)^\ast$ and $(\HH^\I)^\ast$ are MAGs, by \cite[Theorem 3.7]{ARS09}, this can be done by showing that they have the same colliders with order.

Since $(\GG^\I)^\ast$ and $\GG^\I$ differ only by replacing some undirected edges in $(\GG^\I)^\ast$ of the form $\omega_I - i$ for some $I\in\I$ and $i\in I$ with the directed arrows $\omega_I\rightarrow i$, it follows that any discriminating path in $(\GG^\I)^\ast$ is also a discriminating path in $\GG^\I$.  
In particular, all colliders with order in $(\GG^\I)^\ast$ all colliders with order in $(\GG^\I)^\ast$ are also colliders with order in $\GG^\I$. 

Let $\langle i,j,k\rangle$ be a collider with order in $(\GG^\I)^\ast$.
Then it is a collider with order in $\GG^\I$ and hence either a collider with order in $\GG$ or an $\I$-collider with order in $\GG^\I$. 
In the former case, since $\GG$ and $\HH$ are Markov equivalent MAGs, it follows that $\langle i,j,k\rangle$ is also a collider with order in $\HH$.  
In the latter case, since $\GG^\I$ and $\HH^\I$ have the same $\I$-colliders with order, it follows that $\langle i,j,k\rangle$ is also an $\I$-collider with order in $\HH^\I$.  
Since, as we saw in the proof of the forward direction, any discriminating path for a collider with order in $\HH^\I$ is also a discriminating path for a collider with order in $(\HH^\I)^\ast$, it follows that $\langle i,j,k\rangle$ is a collider with order in $(\HH^\I)^\ast$.  
By symmetry of the argument, it follows that $(\GG^\I)^\ast$ and $(\HH^\I)^\ast$ are two MAGs with the same colliders with order.  
Hence, $(\GG^\I)^\ast$ and $(\HH^\I)^\ast$ are Markov equivalent.
It follows that
\[
\J(\GG^\I) = \J((\GG^\I)^\ast) = \J((\HH^\I)^\ast) = \J(\HH^\I), 
\]
and so $\GG$ and $\HH$ are $\I$-Markov equivalent, which completes the proof. 
\hfill$\square$

\section{Lemmata and Proofs for Section~\ref{sec: factorization criteria}} 

The following lemma demonstrates that the collection $\M_\I(\GG)$ constitutes all interventional settings that can arise by intervening at the the targets $\I$ in a distribution Markov to an ADMG $\GG$.
Its proof is analogous to that of \cite[Lemma A.1]{YKU18}.
\begin{lemma}
\label{lem: equating collections}
Suppose $\emptyset\in\I$ and $\GG$ is an ADMG.  
Then $(f^{(I)})_{I\in\I}\in\M_\I(\GG)$ if and only if there exists $f^{(\emptyset)}\in\M(\GG)$ such that for all $I\in\I$, $f^{(I)}$ factorizes according to equation~\eqref{eqn: ADMG interventional factorization}.
\end{lemma}

\begin{proof}
Suppose first that there exists an $f^{(\emptyset)}\in\M(\GG)$ such that for all $I\in\I$, $f^{(I)}$ factorizes according to equation~\eqref{eqn: ADMG interventional factorization}.
It is then immediate from Theorem~\ref{thm: richardson factorization} that $f^{(I)}\in\M(\GG)$ for all $I\in\I$.  
It then remains to check that for all $A\in \A(\GG)$ that $f^{(I)}(x_H\mid x_{\tail(H)}) = f^{(J)}(x_H \mid x_{\tail(H)})$ for all $H\in[A]_\GG$ such that $H\cap(I\cup J) = \emptyset$.  
To see this, note that since $H\cap(I\cup J) = \emptyset$, then $H\cap I = \emptyset$ and $H\cap J = \emptyset$.  
So by equation~\eqref{eqn: ADMG interventional factorization}, 
 \[
 f^{(I)}(x_H\mid x_{\tail(H)}) = f^{(\emptyset)}(x_H\mid x_{\tail(H)}) = f^{(J)}(x_H\mid x_{\tail(H)}). 
 \]
 Hence, $(f^{(I)})_{I\in\I}\in\M_\I(\GG)$.  
 
 Conversely, suppose that $(f^{(I)})_{I\in\I}\in\M_\I(\GG)$.  
 We need to find $f^{(\emptyset)}$ such that $f^{(\emptyset)}\in\M(\GG)$ and for all $I\in\I$, $f^{(I)}$ factorizes according to equation~\eqref{eqn: ADMG interventional factorization}.  
To this end, for $A\in\A(\GG)$ and $H\in[A]_\GG$, pick (if it exists) $I_H\in\I$ such that $I_H\cap H = \emptyset$, and set
 \[
 f^{(\emptyset)}(x_H\mid x_{\tail(H)}) := f^{(I_H)}(x_H\mid x_{\tail(H)}). 
 \]
 Note that this choice is well-defined by the second defining condition of $\M_\I(\GG)$ and the fact that this choice is independent of $A$.  
On the other hand, if no such $I_H$ exists, then set $f^{(\emptyset)}(x_H \mid x_{\tail(H)})$ equal to any strictly positive density.  
Since this operation is well-defined, we get a distribution with density function $f^{(\emptyset)}$ such that for all $A\in\A(\GG)$, 
\[
f^{(\emptyset)}(x_A) = \prod_{H\in[A]_\GG}f^{(\emptyset)}(x_H \mid x_{\tail(H)}).
\]
Hence, by Theorem~\ref{thm: richardson factorization}, $f^{(\emptyset)}\in\M(\GG)$.  
Moreover, for all $I\in\I$ and $A\in \A(\GG)$, 
\begin{equation*}
\begin{split}
f^{(I)}(x_A)
&= \prod_{H\in[A]_\GG}f^{(I)}(x_H \mid x_{\tail(H)}), \\
&= \prod_{H\in[A]_\GG : I\cap H \neq \emptyset}f^{(I)}(x_H \mid x_{\tail(H)})\prod_{H\in[A]_\GG : I\cap H = \emptyset}f^{(I)}(x_H \mid x_{\tail(H)}),\\
&= \prod_{H\in[A]_\GG : I\cap H \neq \emptyset}f^{(I)}(x_H \mid x_{\tail(H)})\prod_{H\in[A]_\GG : I\cap H = \emptyset}f^{(\emptyset)}(x_H \mid x_{\tail(H)}).\\
\end{split}
\end{equation*}
Hence, $f^{(I)}$ factorizes according to equation~\eqref{eqn: ADMG interventional factorization}.
\end{proof}

In \cite[Proposition 3.8]{YKU18}, the authors showed that for a DAG $\GG$, if $\emptyset\in\I$, then $(f^{(I)})_{I\in\I}\in\M_\I(\GG)$ if and only if $(f^{(I)})_{I\in\I}$ satisfies the $\I$-Markov property with respect to $\GG^\I$.  
Key to their proof is to show if $(f^{(I)})_{I\in\I}\in\M_\I(\GG)$, then the conditional factors in
\[
f^{(I)} = \prod_{i\in[p]}f^{(I)}(x_i\mid x_{\pa_\GG(i)})
\]
can be partitioned in such a way that, upon marginalization and conditioning, we can recover property 2 of Definition~\ref{def: ADMG local I-Markov properties}.  
This argument generalizes to ADMGs with the factorization given in Theorem~\ref{thm: richardson factorization}.  
However, it requires that we take a few extra steps to account for the fact that the heads $H\in[A]_\GG$ are no longer necessarily singletons.
In particular, we require the following lemma.

\begin{lemma}
\label{lem: partitioning}
Suppose $(f^{(I)})_{I\in\I}\in\M_\I(\GG)$, we have two sets $B,C\in[p]$, and $I\in\I$ such that $B$ is $m$-separated from $\omega_I$ given $C\cup W_\I\setminus\{\omega_I\}$ in $\GG^\I$.  
Let $A:=\an_\GG(B\cup C)\in\A(\GG)$.  
Let $Z \subset A$ denote all vertices in $A$ that are $m$-connected to $\omega_I$ given $C\cup W_\I\setminus\{\omega_I\}$ in $\GG^\I$, and let $B^\prime := A\setminus(Z\cup C)$.  
Finally, let
\begin{equation*}
\begin{split}
C^\prime &:=\{ i\in C \mid (\pa_\GG(i)\cup\sp_\GG(i))\cap B^\prime\neq\emptyset\}, \mbox{ and}\\
C^{\prime\prime} &:=\{ i\in C \mid (\pa_\GG(i)\cup\sp_\GG(i))\cap B^\prime=\emptyset\}.\\
\end{split}
\end{equation*}
Then
	\begin{enumerate}[1.]
		\item No district in $\GG$ contains both a node of $B^\prime$ and a node of $Z$.
		\item No district in $\GG$ contains both a node of $C^\prime$ and a node of $Z$.
	\end{enumerate}
\end{lemma}

\begin{proof}
For statement (1), suppose on the contrary that there exists a district $D\in\dis(\GG)$ that contains $z\in Z$ and $b^\prime\in B^\prime$.  
By definition of a district, since $z,b^\prime\in D$, we know that they are path-connected in $\GG_{\leftrightarrow}$ by a path consisting of only bidirected arrows.  
Let $\pi = \langle v_0:=z,v_1,\ldots, v_m:=b^\prime\rangle$ denote such a path.  
Then since $z\in A\setminus(B\cup C)$, we know that $z$ is ancestral to either some $b\in B$ or $c\in C$. 
If $z$ is ancestral to $b\in B$, then the directed path from $z$ to $b$ must be blocked by some $c\in C$.  
This is because, by assumption, every $b\in B$ is $m$-separated from $\omega_I$ given $C\cup W_\I\setminus\{\omega_I\}$ in $\GG^\I$.  
Since $z\notin C$ and $z$ is $m$-connected to $\omega_I$ given $C\cup W_\I\setminus\{\omega_I\}$ in $\GG^\I$, concatenating such an $m$-connecting path from $\omega_I$ to $z$ with the directed path from $z$ to $b$ would result in an $m$-connecting path from $\omega_I$ to $b$, which would be a contradiction.  
Hence, $z$ is ancestral to some $c\in C$, and thus there exists some directed path from $z$ to $c$.  
Therefore, $v_1$ is $m$-connected to $\omega_I$ given $C\cup W_\I\setminus\{\omega_I\}$ in $\GG^\I$.  
In other words, $v_0,v_1\in Z$.  
Repeating this argument in an inductive fashion demonstrates that in fact $v_0,\ldots, v_m\in Z$.  
However, this contradicts the assumption that $b^\prime \in B^\prime$.  
Thus, no node in $Z$ can be in the same district as a node in $B^\prime$.   

To see statement (2) holds, we can use a similar inductive argument.  
Suppose on the contrary that there is a district $D\in \dis(\GG)$ containing a node $z\in Z$ and a node $c^\prime\in C^\prime$.  
Since $c^\prime$ and $z$ are in $D$, a district in $\GG$, then they are path-connected by a path $\pi = \langle v_0:=z,v_1,\ldots, v_m:=c^\prime\rangle$ consisting of only bidirected arrows.  
By the argument for statement (1), we know that $\pi$ is an $m$-connecting path given $C\cup W_\I\setminus\{\omega_I,c^\prime\}$.  
Since $c^\prime\in C^\prime$, then there exists $b^\prime \in B^\prime$ and a path $\pi^\prime = \langle v_0:=z,v_1,\ldots, v_m, b^\prime\rangle$ such that $v_m$ is a collider on $\pi^\prime$.  
Hence, $\pi^\prime$ is an $m$-connecting path given $C\cup W_\I\setminus\{\omega_I\}$ in $\GG^\I$, which contradicts the fact that $b^\prime\in B^\prime$.  
Hence, a node of $Z$ and a node of $C^\prime$ cannot be in the same district in $\GG$.
\end{proof}

\begin{remark}
\label{rmk: parents of heads}
By a similar argument as in the proof of Lemma~\ref{lem: partitioning} (see the appendix), we see that the parents of heads $H$ containing elements of $B^\prime$ and/or $C^\prime$ must be contained in $B^\prime\cup C$.  
Similarly, the parents of heads $H$ containing elements of $Z$ and/or $C^{\prime\prime}$ are contained in $Z\cup C^{\prime\prime}$.  
\end{remark}

Given Lemma~\ref{lem: partitioning} and Remark~\ref{rmk: parents of heads}, we can generalize \cite[Proposition 3.8]{YKU18} to ADMGs, as stated in Theorem~\ref{thm: local I-Markov equivalence characterization}.

\subsection{Proof of Theorem~\ref{thm: local I-Markov equivalence characterization} } 
Suppose first that $(f^{(I)})_{I\in\I}$ satisfies the $\I$-Markov property with respect to $\GG^\I$.  
Since $\GG$ is an ADMG, then by Theorem~\ref{thm: richardson factorization} and condition (1) of the $\I$-Markov property, we know that for all $A\in[A]_\GG$
\begin{equation}
\label{eqn: proof factorization}
f^{(I)}(x_A) = \prod_{H\in[A]_\GG}f^{(I)}(x_H\mid x_{\tail(H)}).
\end{equation}
Consider now $H\in[A]_\GG$ for which $H\cap I = \emptyset$. 
Note that since $\GG^\I$ is also an ADMG, and since $H\cap I = \emptyset$, then $\tail_{\GG^\I}(H) = \tail_\GG(H)$.  
Since $\tail_{\GG^\I}(H)$ is the Markov blanket of $H$ in $\GG^\I$ \cite[Section 3.1]{R09}, then $\omega_I$ is $m$-separated from $H$ given $\tail_\GG(H)\cup W_\I\setminus \{\omega_I\}$ in $\GG^\I$.  
So by condition (2) of the $\I$-Markov property, we can replace $f^{(I)}(x_H\mid x_{\tail(H)})$ with $f^{(\emptyset)}(x_H\mid x_{\tail(H)})$ in the product in equation~\eqref{eqn: proof factorization}.  
Since this is independent of the choice of $I\in \I$, Lemma~\ref{lem: equating collections} implies that $(f^{(I)})_{I\in\I}\in\M_\I(\GG)$.  

Conversely, suppose that $(f^{(I)})_{I\in\I}\in\M_\I(\GG)$.  
Notice first that since for all $I\in \I$, we have $f^{(I)}\in\M(\GG)$, then $(f^{(I)})_{I\in\I}$ satisfies condition (1) of the $\I$-Markov property.  
So it remains to check that $(f^{(I)})_{I\in\I}$ satisfies condition (2) of the $\I$-Markov property with respect to $\GG^\I$. 
To this end, assume that $A,B,B^\prime, C,C^\prime, C^{\prime\prime}, Z\subset[p]$ are subsets as specified in Lemma~\ref{lem: partitioning} and Remark~\ref{rmk: parents of heads}, and set $\hat{\I} :=\{\emptyset,I\}$.  
Note, by Lemma~\ref{lem: partitioning} and Remark~\ref{rmk: parents of heads}, we can partition $[A]_\GG$ as 
$
[A]_\GG = Q \sqcup Y \sqcup M \sqcup N,
$
where 
\begin{equation*}
\begin{split}
Q &:=\{ H\in [A]_\GG \mid B^\prime \cap H \neq \emptyset\},\\
Y &:=\{ H\in [A]_\GG \mid Z \cap H \neq \emptyset\},\\
M &:=\{ H\in [A]_\GG \mid B^\prime \cap H = \emptyset, C^\prime \cap H \neq \emptyset\}, \text{ and }\\
N &:=\{ H\in [A]_\GG \mid Z \cap H = \emptyset, C^{\prime\prime} \cap H \neq \emptyset\}.\\
\end{split}
\end{equation*}
By Lemma~\ref{lem: equating collections}, we know that for $\hat{I}\in\hat{\I}$, the marginal density $f^{(\hat{I})}(x_A)$ factors according to equation~\eqref{eqn: ADMG interventional factorization}, and so
\begin{equation*}
\begin{split}
f^{(\hat{I})}(x_A)
&= f^{(\hat{I})}(x_{B^\prime},x_{C^\prime},x_{C^{\prime\prime}},x_Z),\\
&= \prod_{H\in Q}f^{(\hat{I})}(x_H\mid x_{\tail(H)})\prod_{H\in Y}f^{(\hat{I})}(x_H\mid x_{\tail(H)})\\
	&\hspace{1in}\times\prod_{H\in M}f^{(\hat{I})}(x_H\mid x_{\tail(H)})\prod_{H\in N}f^{(\hat{I})}(x_H\mid x_{\tail(H)}),\\
&= \prod_{H\in Q}f^{(\emptyset)}(x_H\mid x_{\tail(H)})\prod_{H\in Y}f^{(\emptyset)}(x_H\mid x_{\tail(H)})\\
&\hspace{1in}\times\prod_{H\in M}f^{(\hat{I})}(x_H\mid x_{\tail(H)})\prod_{H\in N}f^{(\hat{I})}(x_H\mid x_{\tail(H)}),\\
&= g(x_{B^\prime},x_{C^\prime})h(x_Z,x_{C^{\prime\prime}};\hat{I}), 
\end{split}
\end{equation*}
where $g(x_{B^\prime},x_{C^\prime})$ is the product of the conditional factors over $Q$ and $Y$, and $h(x_Z,x_{C^{\prime\prime}};\hat{I})$ is the product of the conditional factors over $M$ and $N$.  
From this point, we can now use an analogous argument to that used in the proof of Proposition 3.8 in \cite{YKU18}.
Namely, marginalizing out the variables for nodes in $Z$ and $B^\prime\setminus B$ gives
\[
f^{(\hat{I})}(x_B,x_C) = \hat{g}(x_B,x_{C^\prime})\hat{h}(x_{C^{\prime\prime}};\hat{I}),
\]
where $\hat{g}(x_B,x_{C^\prime}):= \int_{B^\prime\setminus B}g(x_{B^\prime},x_{C^\prime})$ and $\hat{h}(x_{C^{\prime\prime}};\hat{I}) := \int_{Z}h(x_Z,x_{C^{\prime\prime}};\hat{I})$.  
By conditioning on $C$, it is then straightforward to check that $f^{(\hat{I})}(x_B\mid x_C) = f^{(\emptyset)}(x_B\mid x_C)$ for $\hat{I}\in\hat{\I}$.  
Hence, $(f^{(I)})_{I\in\I}$ satisfies condition (2) of the $\I$-Markov property, completing the proof.
\hfill$\square$

\subsection{Proof of Theorem~\ref{thm: causal alignment}} 
To prove Theorem~\ref{thm: causal alignment}, we use the following lemma.
\begin{lemma}
\cite[Proposition 4]{RERS17}
\label{lem: d and m-separation}
Let $\GG = ([p]\sqcup L,E)$ be a DAG where $[p]\sqcup L$ is a disjoint union.  
For disjoint subsets $A,B,C\subset [p]$, with $C$ possibly empty, $A$ and $B$ are $d$-separated given $C$ in $\GG$ if and only if $A$ and $B$ are $m$-separated given $C$ in the latent projection $\GG([p])$.  
\end{lemma}

Combining Lemma~\ref{lem: d and m-separation} with our Theorem~\ref{thm: local I-Markov equivalence characterization}, we can prove Theorem~\ref{thm: causal alignment}:

\subsubsection{Proof of Theorem~\ref{thm: causal alignment}}
By Theorem~\ref{thm: local I-Markov equivalence characterization}, it suffices to show that $(\tilde{f}^{(I)})_{I\in\I}$ satisfies the $\I$-Markov property with respect to $\GG^\I$.  
Since $f^{(I)}\in\M(\GG)$ for all $I\in\I$ then $\tilde{f}^{(I)}$ is Markov to $\GG([p])$ for all $I\in\I$ \cite[Theorem~7.1]{RS02}.  
Hence, $(\tilde{f}^{(I)})_{I\in\I}$ satisfies the first condition of the $\I$-Markov property stated in Definition~\ref{def: ADMG local I-Markov properties}.
To see that $(\tilde{f}^{(I)})_{I\in\I}$ satisfies the second condition, we must show that whenever $C\cup W_\I\setminus\{w_I\}$ $m$-separates $B$ and $w_I$ in $\GG([p])^\I$ then $\tilde{f}^{(\emptyset)}(X_B\mid X_C) = \tilde{f}^{(I)}(X_B\mid X_C)$. 

To this end, notice first that since $B,C\subset[p]$ we have for all $I\in\I$ that
\begin{equation*}
\begin{split}
f^{(I)}(X_B,X_C) &= \int_{L\cup[p]\setminus(B\cup C)}f^{(I)} = \int_{[p]\setminus(B\cup C)}\int_Lf^{(I)} = \tilde{f}^{(I)}(X_B,X_C),	\\
\end{split}
\end{equation*}
and similarly $f^{(I)}(X_C) = \tilde{f}^{(I)}(X_C)$.
It then follows that 
\begin{equation}
\label{eqn: latent}
f^{(I)}(X_B\mid X_C) = \tilde{f}^{(I)}(X_B\mid X_C)
\end{equation}
for all $I\in\I$.
Moreover, since $\GG^\I$ is a DAG and $\GG^\I([p]) = \GG([p])^\I$, then by Lemma~\ref{lem: d and m-separation} we know that $B$ is $d$-separated from $w_I$ given $C\cup W_\I\setminus\{w_I\}$ in $\GG^\I$ if and only if $B$ is $m$-separated from $w_I$ given $C\cup W_\I\setminus\{w_I\}$ in $\GG([p])^\I$.
Since $(f^{(I)})_{I\in\I}\in\M_\I(\GG)$, it follows from \cite[Proposition~3.8]{YKU18} that $f^{(\emptyset)}(X_B\mid X_C) = f^{(I)}(X_B\mid X_C)$ whenever $B$ is $d$-separated from $w_I$ given $C\cup W_\I\setminus\{w_I\}$ in $\GG^\I$. 
So by Lemma~\ref{lem: d and m-separation} and \eqref{eqn: latent}, we conclude that $\tilde{f}^{(\emptyset)}(X_B\mid X_C) = \tilde{f}^{(I)}(X_B\mid X_C)$ whenever $B$ is $d$-separated from $w_I$ given $C\cup W_\I\setminus\{w_I\}$ in $\GG([p])^\I$. 
This completes the proof.
\hfill$\square$

\subsection{Proof of Theorem~\ref{thm: coincide}} 
The equivalence of $(1)$ and $(2)$ was shown in Theorem~\ref{thm: main}.  
So it only remains to show the equivalence of $(1)$ and $(3)$.
Suppose first that $\GG$ and $\HH$ are $\I$-Markov equivalent.  
Then $\GG^\I$ and $\HH^\I$ have the same set of $m$-separation statements.  
Hence, by Theorem~\ref{thm: local I-Markov equivalence characterization}, $\M_\I(\GG) = \M_\I(\HH)$.  

Conversely, suppose now that $\GG$ and $\HH$ on node set $[p]$ are not $\I$-Markov equivalent.  
It follows that (without loss of generality) there exists an $m$-separation statement $\langle A,B \mid C \rangle \in\J(\GG^\I)$ that is not contained in $\J(\HH^\I)$.  
By \cite[Lemma 5]{KJSB19}, it follows that at least one of the following is true:
\begin{enumerate}
	\item There exist $X,Y,Z\subset [p]$ such that $\langle X,Y \mid Z, W_\I\rangle \in \J(\GG^\I)$ and $\langle X,Y \mid Z, W_\I\rangle\notin\J(\HH^\I)$, and/or
	\item there exist $T,S\subset [p]$ and $\omega_I\in W_\I$ such that $\langle \omega_I,T \mid S,W_\I\setminus\{\omega_I\}\rangle\in\J(\GG^\I)$ and $\langle \omega_I,T \mid S,W_\I\setminus\{\omega_I\}\rangle\notin\J(\HH^\I)$. 
\end{enumerate}
It is important to note that, while the augmented graph used by \cite{KJSB19} is different than our $\I$-graphs $\GG^\I$, their proof of Lemma $5$ is independent of the choice of nodes augmenting the graph (i.e., the choice for the set $W_\I$).  
Thus, we can readily apply it to our situation here as well.  
Therefore, we can now assume that the CI relation in question $\langle A,B \mid C \rangle$ is of one of the two forms given in (1) and (2).  
Now, by Theorem~\ref{thm: local I-Markov equivalence characterization}, since $\M_\I(\GG) = \M_\I(\HH)$ is equivalent to the $\I$-Markov property holding, we would like to show that there exists an interventional setting $(f^{(I)})_{I\in\I}$ that is $\I$-Markov to $\HH^\I$ but not $\I$-Markov to $\GG^\I$.  
To do so, we can use the same construction as in the proof of \cite[Lemma 4]{KJSB19}.  
For the sake of completeness, we recall some of these details now.  

Namely, let us first suppose that we are in case (1), and there is a CI relation $\langle X,Y \mid Z\rangle\in\J(\GG^\I)$ that is not in $\J(\HH^\I)$ and that $X,Y,Z\subset [p]$. 
Since $W_\I$ contains only source nodes, then $\langle X,Y \mid Z\rangle\in\J(\GG)$ and $\langle X,Y \mid Z\rangle \notin\J(\HH)$.  
Since there always exist distributions faithful to a given ancestral graph (see, for instance, \cite{S13}), we can pick an observational distribution $f^{(\emptyset)}$ that is faithful to $\HH$.  
Any sequence of interventions targeting $\I$, say $(f^{(I)})_{I\in\I}$, with respect to the observational density $f^{(\emptyset)}$ will then belong to $\M_\I(\HH)$ but not $\M_\I(\GG)$.  
This is because $f^{(\emptyset)}\notin\M(\GG)$.  
Hence, $\M_\I(\GG) \neq \M_\I(\HH)$.  

Suppose now that we are in case (2).  
Then there exists $T,S\subset [p]$ and $\omega_{\hat{I}}\in W_\I$ such that $\langle \omega_{\hat{I}},T \mid S,W_\I\setminus\{\omega_{\hat{I}}\}\rangle\in\J(\GG^\I)$ and $\langle \omega_{\hat{I}},T \mid S,W_\I\setminus\{\omega_{\hat{I}}\}\rangle\notin\J(\HH^\I)$. 
In this case, we can use the construction from the proof of \cite[Lemma 4]{KJSB19}, which is a natural generalization of the construction used in the proof of \cite[Lemma 3.10]{YKU18}.  
Here, we refer the reader to the proof of \cite[Lemma 4]{KJSB19} for complete details, and we simply note that it is sufficient to apply their techniques with respect to $I := \hat{I}$ and $J := \emptyset$.  
The result is an interventional setting $(f^{(I)})_{I\in\I}\in\M_\I(\HH)$ such that for $\hat{I}\in\I$ we have $f^{(\emptyset)}(X\mid Z) \neq f^{(\hat{I})}(X \mid Z)$.  
Since $\langle \omega_{\hat{I}},T \mid S,W_\I\setminus\{\omega_{\hat{I}}\}\rangle\in\J(\GG^\I)$, it follows that $(f^{(I)})_{I\in\I}\notin\M_\I(\GG)$.  
Hence, $\M_\I(\GG)\neq\M_\I(\HH)$, which completes the proof. 
\hfill$\square$

\end{document}